\documentclass[a4paper,11pt,reqno]{amsart}

\usepackage[a4paper,margin=3cm]{geometry}
\usepackage{microtype} 

\usepackage{mathtools,amssymb,amsmath,amsthm}
\usepackage{xfrac}
\usepackage{stmaryrd,mathrsfs,bm,yfonts}
\usepackage{esint}
\usepackage[dvipsnames]{xcolor}
\usepackage[pagebackref,colorlinks=true,linkcolor=MidnightBlue,citecolor=Black,urlcolor=MidnightBlue]{hyperref}
\usepackage[capitalise,nameinlink]{cleveref}
\usepackage{doi}
\usepackage[utf8]{inputenc}
\usepackage{tikz,tikz-3dplot}

\newtheorem{theorem}{Theorem}[section]
\newtheorem{lemma}[theorem]{Lemma}

\theoremstyle{definition}

\newtheorem{remark}[theorem]{Remark}

\numberwithin{equation}{section}

\newcommand{\R}{\mathbb{R}}

\def\x#1{{\rm (\ref{#1})}}

\renewcommand{\epsilon}{\varepsilon}

\def\XXint#1#2#3{{\setbox0=\hbox{$#1{#2#3}{\int}$ }
  \vcenter{\hbox{$#2#3$ }}\kern-.6\wd0}}

\hypersetup{
  pdfauthor={Name}, 
  pdftitle={Plain Title}
}

\makeatletter
\renewcommand{\tocsection}[3]{%
  \indentlabel{\@ifnotempty{#2}{\bfseries\ignorespaces\makebox[\@ifempty{#1}{25pt}{75pt}][l]{#1 #2\quad}}}\bfseries#3}
\renewcommand{\tocsubsection}[3]{%
  \indentlabel{\@ifnotempty{#2}{\ignorespaces\makebox[30pt][l]{#1 #2\quad}}}#3}
\newcommand\@dotsep{4}
\def\@tocline#1#2#3#4#5#6#7{\relax
  \ifnum #1>\c@tocdepth \else
    \par \addpenalty\@secpenalty\addvspace{#2}%
    \begingroup \hyphenpenalty\@M
    \@ifempty{#4}{\@tempdima\csname r@tocindent\number#1\endcsname\relax}{\@tempdima#4\relax}%
    \parindent\z@ \leftskip#3\relax \advance\leftskip\@tempdima\relax
    \rightskip\@pnumwidth plus1em \parfillskip-\@pnumwidth
    #5\leavevmode\hskip-\@tempdima{#6}\nobreak
    \leaders\hbox{$\m@th\mkern \@dotsep mu\hbox{.}\mkern \@dotsep mu$}\hfill
    \nobreak
    \hbox to\@pnumwidth{\@tocpagenum{\ifnum#1=1\bfseries\fi#7}}\par
    \nobreak
    \endgroup
  \fi}
\AtBeginDocument{%
\expandafter\renewcommand\csname r@tocindent0\endcsname{0pt}}
\def\l@section{\@tocline{1}{0pt}{1pc}{25pt}{}}
\def\l@subsection{\@tocline{2}{0pt}{2pc}{30pt}{}}
\makeatother

\title[A new proof  on quasilinear Schr\"{o}dinger equations]{ A new proof  on quasilinear Schr\"{o}dinger equations with prescribed mass and combined nonlinearities}

\author[J. H. Chen]{Jianhua Chen}
\address[Jianhua Chen]{School of Mathematics and Computer Sciences, Nanchang University,  Nanchang, Jiangxi 330031, P. R. China}
\email{cjh19881129@163.com}

\author[J. J. Sun]{Jijiang Sun}
\address[Jijiang Sun]{School of Mathematics and Computer Sciences, Nanchang University,  Nanchang, Jiangxi 330031, P. R. China}
\email{sunjijiang2005@163.com}

\author[C. G. Yuan]{Chenggui Yuan}
\address[Chenggui Yuan]{Department of Mathematics, Swansea University, Bay Campus, Swansea SA1 8EN, UK}
\email{C.Yuan@swansea.ac.uk}

\author[J. Zhang]{Jian Zhang}
\address[Jian Zhang]{College of Science, Hunan University of Technology and Business, Changsha, Hunan 410205, P. R. China}
\email{zhangjian@hutb.edu.cn}

\begin{document}

\maketitle

\begin{abstract}
In this work,  we study the quasilinear Schr\"{o}dinger equation
\begin{equation*}
\aligned
-\Delta u-\Delta(u^2)u=|u|^{p-2}u+|u|^{q-2}u+\lambda u,\,\, x\in\R^N,
\endaligned
\end{equation*}
under the mass constraint
\begin{equation*}
\int_{\R^N}|u|^2\text{d}x=a,
\end{equation*}
where $N\geq2$, $2<p<2+\frac{4}{N}<4+\frac{4}{N}\leq q<2\cdot2^*$, $a>0$ is a given mass and $\lambda$ is a Lagrange multiplier. As a continuation of our previous work (Chen et al., 2025, arXiv:2506.07346v1),
we  establish some results  by means of a suitable change of variables as follows:
  \begin{itemize}
\item[{\bf(i) }] {\bf qualitative analysis of the constrained minimization}\\
For $2<p<4+\frac{4}{N}\leq q<2\cdot2^*$, we provide a detailed study of the minimization problem under some appropriate conditions on $a>0$;
 \end{itemize}
  \begin{itemize}
\item[{\bf(ii)}]{\bf existence of two  distinct radial normalized solutions}\\
For $2<p<2+\frac{4}{N}<4+\frac{4}{N}<q\leq2^*$,  we obtain a radial local minimizer under the normalized constraint;\\
For $2<p<2+\frac{4}{N}<4+\frac{4}{N}<q\leq2^*$,   we obtain a radial mountain pass type normalized solution distinct from the local minimizer.
 \end{itemize}
Notably, the second result {\bf (ii)} resolves the open problem {\bf(OP1)} posed by (Chen et al., 2025, arXiv:2506.07346v1). Unlike previous approaches that rely on constructing Palais-Smale-Pohozaev sequences by [Jeanjean, 1997, Nonlinear Anal. {\bf 28}, 1633-1659], we obtain the mountain pass solution employing a new method, which relies on the monotonicity trick developed by (Chang et al., 2024, Ann. Inst. H. Poincar\'{e} C Anal. Non Lin\'{e}aire, {\bf 41}, 933-959).

We emphasize that the methods developed in this work can be extended to investigate the existence of mountain pass-type normalized solutions for other classes of quasilinear Schr\"{o}dinger equations via dual method.
\end{abstract}

\tableofcontents

\section{Introduction}
In this paper, we
are concerned with the following quasilinear Schr\"{o}dinger equation
\begin{equation}\label{p}
\aligned
-\Delta u-\Delta(u^2)u=|u|^{p-2}u+|u|^{q-2}u+\lambda u,\ \ \ x\in\R^N,
\endaligned
\end{equation}
with the prescribed mass
\begin{equation*}
  \int_{\R^N}u^2\text{d}x=a,
\end{equation*}
where $2<p<2+\frac{4}{N}<4+\frac{4}{N}\leq q<2\cdot2^*$, $\lambda$ is a Lagrange multiplier,  $a$ is a given mass,  and $N\geq2$.

To the best of our knowledge, the study of standing wave solutions for quasilinear Schr\"{o}dinger equations dates back to the seminal work of \cite{PSW}. Using constrained minimization methods, the authors of \cite{PSW} established existence results under various conditions. Notably, in their proofs, the non-differentiability of the energy functional for \x{p} did not play a significant role, and thus classical minimax techniques were not employed in this context.
To address this challenge, the authors \cite{CL1,LWW} introduced a novel dual method, transforming the quasilinear problem into a semilinear one and restoring the differentiability of the energy functional for \x{p}. This breakthrough enabled the application of classical minimax theory to quasilinear problems. Over the past two decades, this approach has led to extensive research on the existence of solutions for \x{p}, and we refer the reader to \cite{CTC,CTC2,DMS2,JS,DMS,FS,HQZ,LW,LWW1,LLW6} for further developments.
In a different direction, Liu et al. \cite{LLW2} proposed an innovative perturbation method. By introducing an auxiliary term to the energy functional associated with \x{p}, they restored differentiability within a constrained functional space. Using classical minimax theory, they first obtained critical points for the perturbed functional and then studied their asymptotic behavior to derive solutions for the original problem. This framework has inspired numerous subsequent works on standing wave solutions.  We refer the readers to \cite{LLW3,LLW4}  and its references therein.

In recent years, since the
$L^2$-norm plays a crucial role in investigating the orbital stability or instability of solutions,  the study of normalized solutions is strongly motivated by physical considerations. Based on this fact, the normalized solutions of the quasilinear Schr\"{o}dinger equation \x{p} with a prescribed
$L^2$-norm has been widely studied. Consequently, the existence problem for normalized solutions carries substantial interest from both physical and mathematical viewpoints. Within the variational framework, our objective is to obtain a nontrivial solution
$(u_a,\lambda_a)\in \mathcal{K}\times \R$ to \x{p} such that $(u_a,\lambda_a)$ satisfying $\|u_a\|_2^2=a$, where
$$
\mathcal{K}:=\left\{u\in H^1(\R^N): \int_{\R^N}|\nabla u|^2u^2\text{d}x<+\infty\right\}.
$$
Essentially, solving this problem reduces to identifying the critical points of the corresponding Euler-Lagrange functional
\begin{equation*}
\mathcal{I}_\lambda(u)=\frac{1}{2}\int_{\R^N}(1+2u^2)|\nabla u|^2\text{d}x-\frac{\lambda}{2}\int_{\R^N}|u|^2\text{d}x-\frac{1}{p}\int_{\R^N}|u|^p\text{d}x-\frac{1}{q}\int_{\R^N}|u|^q\text{d}x,
\end{equation*}
where $u\in\mathcal{K}$. The weak solution characterization for \x{p} asserts that $u$ satisfies the equation if and only if
\begin{equation*}
0=\langle \mathcal{I}'_\lambda(u),\phi\rangle=\lim_{t\rightarrow0^+}\frac{\mathcal{I}_\lambda(u+t\phi)-\mathcal{I}_\lambda(u)}{t},
\end{equation*}
for any $\phi\in\mathcal{C}^\infty_0(\R^N)$. Using this method, Colin et al. \cite{CJS1} and Jenajean et al. \cite{LJ} firstly studied the following quasilinear problem
\begin{equation}\label{p1}
\aligned
-\Delta u-\Delta(u^2)u=|u|^{p-2}u+\lambda u,\ \ \ x\in\R^N,
\endaligned
\end{equation}
with the prescribed mass
$
  \int_{\R^N}u^2\text{d}x=a,
$
and define the minimization
$
m_0(a)=\inf\limits_{u\in S_a}I_0(u),
$
where
\begin{equation*}
I_0(u)=\frac{1}{2}\int_{\R^N}(1+2u^2)|\nabla u|^2\text{d}x-\frac{1}{p}\int_{\R^N}|u|^p\text{d}x
\end{equation*}
and
\begin{equation*}
S_a:=\left\{u\in \mathcal{K}\bigg|\int_{\R^N}u^2\text{d}x=a\right\}.
\end{equation*}
By constrained minimization method, they showed that $I_0$ is bounded below and $m_0(a)$ can be achieved when $2<p<4+\frac{4}{N}$, and  $I_0$ is unbounded below  when $4+\frac{4}{N}<p<22^*$. Thus $4+\frac{4}{N}$ is called as $L^2$-critical exponent just as  $2+\frac{4}{N}$ is $L^2$-critical exponent for semilinear Schr\"{o}dinger equations.
For $L^2$-critical problem, Ye et al. \cite{YY}  reduced the constraint set $S_a$ to
$$
N_a=\left\{u\in S_a\big|\int_{\R^N}u^2|\nabla u|^2\text{d}x<\frac{1}{4+\frac{4}{N}}\int_{\R^N}|u|^{4+\frac{4}{N}}\text{d}x\right\}
$$
and showed the existence of normalized solution for this problem for $N\leq3$ by a minimization argument.  To use minimax theory, Jeanjean et al. \cite{JLW} were the first to use the perturbation  method developed in \cite{LLW2} to show the existence of  multiple solutions for $L^2$-subcritical growth.
For $L^2$-supercritical problem, Li et al. \cite{LZ} established the existence of ground state normalized solutions and infinitely many normalized solutions for \x{p1} via perturbation  method  when $N=2,3$. Furthermore,
the existence of normalized solutions for the quasilinear Schr\"{o}dinger equation with potentials via the  perturbation  method  can be found in \cite{GG}.
Very recently, Jeanjean et al. \cite{JZZ} used the constrained minimization method to  extend  the analysis from $1\leq N\leq3$ to arbitrary
$N\geq1$. For problem \x{p} with combined nonlinearities, in \cite{ML}, Mao et al. showed that equation \x{p} has a local minimizer and a mountain pass type normalized solution for $2<p<2+\frac{4}{N}<4+\frac{4}{N}<q<2^*$. Afterwards, He et al. \cite{HQZ1}  adopted the minimization argument and some strategies in \cite{JZZ,NS}, and extended the results for the dimension $1\leq N\leq3$ to $1\leq N\leq4$. Meanwhile, in \cite{LYC}, Li et al. showed the existence of normalized solutions with Sobolev critical growth.

For the two approaches discussed above, several  drawbacks can be identified as follows:

{\bf the constrained minimization method}
 \begin{itemize}
\item The non-differentiability of the energy functionals prevents the applications of standard minimax methods. Consequently, solutions beyond the ground state cannot be obtained through these techniques.
\end{itemize}

{\bf  the  perturbation  method }
 \begin{itemize}
\item The variational discrepancy between the perturbed and original energy functionals implies that their respective ground state energies on the manifold are generally non-equivalent. This fundamental difference precludes the derivation of ground state normalized solutions through manifold-based perturbation methods. Furthermore, although the perturbation framework enables effective deployment of minimax principles, its implementation involves substantial technical subtleties that may engender complications.
\end{itemize}
Based on the above facts, the dual approach seems more capable of avoiding these issues. By the dual method, the quasilinear problem can be transformed into semilinear problem. Through this idea, Zhang et al. \cite{ZLW} first  used the dual method to prove the existence of infinitely
many normalized solutions and  minimization problem  for \x{p1} with $L^2$-subcritical growth and a new constraint condition.
In those methods, since $v_t=t^{N/2}v(tx)$($t>0$) dose not keep the constraint
unchanged,  they can only
analysed  the equivalence relations  between the constraint conditions to prove the existence of normalized solutions.
To overcome this obstacle, the authors \cite{CRSZ}  used a novel stretching mapping
\begin{equation*}
v_t(x):=f^{-1}(t^{N/2}f(v(tx)))(t>0)
\end{equation*}
and established a new variational framework to study some minimization problems. By means of this stretching mapping, it can keep the constraint
unchanged.  The main features of \cite{CRSZ} is that they considered the general $L^2$-supercritical growth and critical exponential growth for $N=2,3$.
 Note that the authors \cite{CRSZ} proposed the following open question:

{\bf(OP1)} Existence of mountain pass type normalized solutions via the dual method:
  \begin{itemize}
\item Can the current technical limitations be overcome?

\item What new tools would be required?
 \end{itemize}
The mainly aim of this paper is to answer the above open question {\bf(OP1)}. In particular, we first analyse the  minimization problem for \x{p} with different conditions on $p,q$. We then show the existence of two distinct normalized solutions, one of which is a  local minimizer  and the other is a mountain pass solution on the different range for $p,q$.

 Before stating our results, we show the following Gagliardo-Nirenberg inequality \cite{MIW}, which will play a fundamental role in our subsequent analysis.
\begin{lemma}\label{lemma1.1}(\cite{MIW})
 For any $N\geq2$ and $s\in(2,2^*)$, there is a constant $C_{N,s}>0$
 depending on $N$ and $s$ such that
 \begin{equation*}
   \int_{\R^N}|u|^s\text{d}x\leq C_{N,s}^s\left( \int_{\R^N}|u|^2\text{d}x\right)^{\frac{2s-N(s-2)}{4}}\left( \int_{\R^N}|\nabla u|^2\text{d}x\right)^{\frac{N(s-2)}{4}},\ \ \forall\,u\in H^1(\R^N).
 \end{equation*}
\end{lemma}

\par
Now, we give another version of Gagliardo-Nirenberg inequality as follows. Let
\begin{equation*}
\mathcal{E}:=\left\{u\in L^1(\R^N): |\nabla u|\in L^2(\R^N)\right\}
\end{equation*}
with the norm
\begin{equation*}
\|u\|_\mathcal{E}=\|\nabla u\|_2+\|u\|_1.
\end{equation*}
\begin{lemma}\label{lemma1.2}(\cite{MA})
 For any $N\geq2$ and $t\in(2,2\cdot2^*)$, there is a constant $C_{N,t}>0$
 depending on $N$ and $t$ such that
 \begin{equation*}
   \int_{\R^N}|u|^{\frac{t}{2}}\text{d}x\leq C_{N,t}^t\left( \int_{\R^N}|u|\text{d}x\right)^{\frac{4N-t(N-2)}{2(N+2)}}\left( \int_{\R^N}|\nabla u|^2\text{d}x\right)^{\frac{N(t-2)}{2(N+2)}},\ \ \forall\,u\in\mathcal{E}.
 \end{equation*}
\end{lemma}

To state our results, we give the following definition. Applying an argument developed by Liu-Wang-Wang \cite{LWW} and Colin-Jeanjean \cite{CL1}, we need to recall some definitions again and collect these as follows. At first, we make the change of variables by $v=f^{-1}(u)$, where $f$ is defined by
\begin{equation}\label{function}
f'(t)=\frac{1}{(1+2f^2(t))^{\frac{1}{2}}}\,\,\text{on}\,\,[0,\infty)\,\,\text{and}\,\,f(t)=-f(-t)\,\,\text{on}(-\infty,0].
\end{equation}
Then equation \x{p} in form can be transformed into
\begin{equation*}
-\Delta v=|f(v)|^{p-2}f(v)f'(v)+|f(v)|^{q-2}f(v)f'(v)+\lambda f(v)f'(v),\ \ \ x\in\R^N,
\end{equation*}
with the prescribed mass
$$
\int_{\R^N}|f(v)|^2=a.
$$

Let $\theta\in[1/2,1]$.
Now, we establish the following problem
\begin{equation}\label{subsection1}
-\Delta v-\lambda f(v)f'(v)=\theta|f(v)|^{p-2}f(v)f'(v)+\theta|f(v)|^{q-2}f(v)f'(v),\,\, x\in\R^N,
\end{equation}
with prescribed mass
$$
\int_{\R^N}|f(v)|^2\text{d}x=a.
$$
To this end,
 we define the following family of $\mathcal{C}^2$-functional for  problem \x{subsection1}
\begin{equation}\label{e}
\aligned
\Phi_{\theta}(v)&=\frac{1}{2}\int_{\R^N}|\nabla v|^2\text{d}x-\frac{\theta}{p}\int_{\R^N}|f(v)|^p\text{d}x-\frac{\theta}{q}\int_{\R^N}|f(v)|^{q}\text{d}x
\endaligned
\end{equation}
on the constraint
\begin{equation*}
\mathcal{S}_a:=\left\{v\in H^1(\R^N): \int_{\R^N}|f(v)|^2\text{d}x=a\right\}.
\end{equation*}

Set
$
A(v):=\frac{1}{2}\int_{\R^N}|\nabla v|^2\text{d}x
$
 and
 \begin{equation*}
B(v):=\frac{1}{p}\int_{\R^N}|f(v)|^p\text{d}x+\frac{1}{q}\int_{\R^N}|f(v)|^{q}\text{d}x.
 \end{equation*}
Thus $A(v)\rightarrow+\infty$ as $\|u\|\rightarrow+\infty$, $B(v)\geq0$ for any $v\in H^1(\R^N)$ and $\Phi'_{\theta}$, $\Phi''_{\theta}$ are $\alpha$-H\"{o}lder continuous on bounded sets for some $\alpha\in(0,1]$.

 To look for such  solutions,
by  the definition of \x{e}, we have
\begin{equation}\label{e1}
\Phi_{\theta}(v_t)=\frac{t^2}{2}\int_{\R^N}\frac{1+2t^Nf^2(v)}{1+2f^2(v)}|\nabla v|^2\text{d}x- \theta t^{-N}\int_{\R^N}\left[\frac{1}{p}\left|t^{N/2}f(v)\right|^{p}+\frac{1}{q}\left|t^{N/2}f(v)\right|^{q}\right]\text{d}x.
\end{equation}
By $\frac{\text{d}\Phi_\theta(v_t)}{\text{d}t}\bigg|_{t=1}=0$ and \x{e1}, it follows that
\begin{equation}\label{e2}
\aligned
\mathcal{P}_{\theta}(v):&=\int_{\R^N}|\nabla v|^2\text{d}x+\frac{N}{2}\int_{\R^N}\frac{2f^2(v)}{1+2f^2(v)}|\nabla v|^2\text{d}x\\
&\quad\quad\quad-\frac{N(p-2)\theta}{2p}\int_{\R^N}|f(v)|^p\text{d}x-\frac{N(q-2)\theta}{2q}\int_{\R^N}|f(v)|^{q}\text{d}x=0.
\endaligned
\end{equation}

In this paper, we set $\Phi(v):=\Phi_1(v)$.
Next, we establish our main results. At first, we define the following minimization problem by
\begin{equation*}
 m(a):=\inf_{v\in\mathcal{S}_a}\Phi(v).
\end{equation*}

\begin{theorem}\label{theorem1.1}
Assume that $N\geq2$ and $2<p<4+\frac{4}{N}\leq q<2\cdot2^*$.
Then there hold:

$(i)$ if $2<p<4+\frac{4}{N}<q<2\cdot2^*,$ then $m(a)=-\infty$ for all $a>0$;

$(ii)$ if $2<p<2+\frac{4}{N}<4+\frac{4}{N}=q<2\cdot2^*$, then there exists a constant $\bar{a}^*_N>0$ such that
  $-\infty<m(a)<0$ is achieved for all $0<a<\bar{a}^*_N$.
\end{theorem}

\begin{remark}
Although the existence of normalized solutions for quasilinear Schr\"{o}dinger equations with combined nonlinearities has been studied in \cite{ML}, a more refined analysis of the minimization problem was not conducted. But in this paper, we give a refined analysis on the minimization problem and give the existence of normalized solution with $2<p<2+\frac{4}{N}<4+\frac{4}{N}=q<2\cdot2^*$. To our knowledge, this result has never been considered before.
\end{remark}

Now, we give the second result of  this paper.
\begin{theorem}\label{theorem1.2}
Assume that $N\geq2$ and $2<p<2+\frac{4}{N}<4+\frac{4}{N}<q\leq2^*$.
Then there exists $a^*_N>0$ such that for all $a\in(0,a^*_N)$,

(i) problem \x{p} has  a radial local minimizer;

(ii) problem \x{p} admits  a radial mountain pass type normalized solution distinct from the local minimizer.
\end{theorem}

\begin{remark}
(i) As pointed out in \cite{CRSZ}, to seek a mountain pass type normalized solution, the general method is to find the identity relationship between the energy functional and the Pohazaev functional. But in this place,  the identity relationship between the energy functional and the Pohazaev functional is not very clear, since $f$ is nonlinear. So it is very difficult for us to construct  a bounded Palais-Smale-Pohozaev sequence as \cite{Chen1,Chen3,JSL,Jean}. To overcome this  difficult, we use the the monotonicity trick  developed in \cite{BCJS,CJS,Jean1,SZ,ZP} to get a almost every bounded Palais-Smale sequence for a family of $\mathcal{C}^2$-functional.

(ii) The results of Theorem \ref{theorem1.2} were already proved in \cite{HQZ1,ML}. Note that the authors \cite{ML} mainly use a perturbation method, which is very complex,  especially in limit analysis and $L^\infty$-estimation. Moreover, the authors \cite{HQZ1} mainly  adopt a new method, which depends on
 the implicit function theorem and fiber mapping.  In this paper, we present a new proof based on the dual method  and Lagrange Multipliers' rule. Compared with \cite{HQZ1,ML}, our approach is simpler, since it does not rely on   $L^\infty$-estimation, the implicit function theorem or similar technical tools.

(iii)
The existence of mountain pass type normalized solutions for the range $2<p<2+\frac{4}{N}<2^*<q<22^*$
remains an open problem under the dual method framework. The primary obstacle is proving that the Lagrange multiplier is negative, which is a critical requirement to establish the compactness of Palais-Smale sequences.

(iv) The results of this paper  can be regarded as a further study and outlook on the literature \cite{CRSZ}.

(v) In Theorem  \ref{theorem1.2}, since $4+\frac{4}{N}<q\leq2^*$, it shows that $2\leq N\leq 3$.
\end{remark}

 \vskip2mm
  \par
The paper is organized as follows.  In section 2, we give some preliminary lemmas. In section 3, we will prove Theorem  \ref{theorem1.1}.  In section 4, we shall prove Theorem  \ref{theorem1.2}.
  \vskip2mm
  \par\noindent

  \vskip2mm
  \par\noindent
{\bf Notation:} Throughout this paper, the notations we need to use are as follows:

\quad $\bullet$ $\mathcal{X}$ is a Banach Space and $\mathcal{X}^*$ denotes the dual space of $\mathcal{X}$.

\quad $\bullet$ $L^q(\R^N)$ denotes the Lebesgue space with the norm
$$
\|u\|_q=\left(\int_{\R^N}|u|^q \text{d}x\right)^\frac{1}{q}
$$
 for $q\in[1,+\infty)$.

\quad $\bullet$
$
2^*=\left\{
\begin{array}{ll}
\frac{2N}{N-2}, & \textrm{if $N\geq3$},\\
+\infty,& \textrm{if $N=1,2$}.
\end{array} \right.\\
$

\quad $\bullet$ $C$ denote various positive constants which may vary from line to line.

\quad $\bullet$
 Let $
H^{1}(\R^N)=\left\{u\in L^{2}(\R^N) : |\nabla u|\in L^2(\R^N)\right\}$
with the norm $$\|u\|=\left(\int_{\R^N}(|\nabla u|^2+u^2)\text{d}x\right)^{\frac{1}{2}}.$$

\quad $\bullet$
Define
$
H^{1}_r(\R^N)=\left\{u\in H^{1}(\R^N) : u(x)=u(|x|)\ \ \text{a.e. in}\ \ \R^N\right\}.$

\section{Preliminary lemmas}
In this section, we need to give some preliminary results. At first,
let us recall some properties of the change of variables $f : \R\rightarrow\R$,  which are proved in \cite{CL1,LWW}.
\begin{lemma}\label{lemma2.1}(\cite{CL1,LWW})
 The function  $f(t)$ defined by \x{function} and its derivative have the following properties: \\
$(1)$ $f$  is uniquely defined, $\mathcal{C}^\infty$  and invertible;\\
$(2)$ $|f'(t)|\leq1$   for all  $t\in\R$;\\
$(3)$ $|f(t)|\leq|t|$   for all  $t\in\R$;\\
$(4)$ $f(t)/t\rightarrow1$  as  $t\rightarrow0$;\\
$(5)$ $f(t)/\sqrt{t}\rightarrow2^{\frac{1}{4}}$ as $t\rightarrow+\infty$;\\
$(6)$ $f(t)/2\leq tf'(t)\leq f(t)$ for all  $t>0$;\\
$(7)$ $f^2(t)/2\leq tf(t)f'(t)\leq f^2(t)$   for all  $t\geq0$;\\
$(8)$ $|f(t)|\leq2^{1/4}|t|^{1/2}$  for all $t\in\R$;\\
$(9)$ {\it there exists a positive constant} $C$ {\it such that}
\begin{displaymath}
|f(t)|\geq\left\{
\begin{array}{ll}
C|t|, & \textrm{if $|t|\leq 1$},\\
C|t|^{\frac{1}{2}},& \textrm{if $|t|\geq1$};
\end{array} \right.\\
\end{displaymath}
$(10)$ $|f(t)f'(t)|\leq1/\sqrt{2}$  for all $t\in\R$.
 \end{lemma}

 \begin{lemma}\label{lemma2.3}(\cite{CRSZ})
 (i) For any $v\in H^1(\R^N)$,  there holds $f(v)\in H^1(\R^N)$.

(ii) The mapping $v\mapsto f(v)$ from $H^1(\R^N)$  into $L^q(\R^N)$ is continuous for $q\in[2,22^*]$.

(iii) The mapping $v\mapsto f(v)$ from $H^1(\R^N)$  into $H^1(\R^N)$  is continuous.
\end{lemma}

By   \cite{FS,XW}, we can obtain the following lemma.
\begin{lemma}\label{lemmas2.4}
For any $v\in H^1(\R^N)$, there exists $C>0$ such that
\begin{equation*}
\int_{\R^N}(|\nabla v|^2+|f(v)|^2)\text{d}x\geq C\int_{\R^N}(|\nabla v|^2+v^2)\text{d}x.
\end{equation*}
 \end{lemma}

 \begin{lemma}\label{lemmas2.1}(\cite{CRSZ})
Any critical point $v$ of $\Phi_\theta|'_{\mathcal{S}_a}$ satisfies $\mathcal{P}_\theta(v)=0$.
\end{lemma}

To obtain a Palais-Smale sequence at mountain pass level, we will use the monotonicity trick on the family of functionals, which was first proposed by Struwe \cite{MS1,MS2} to solve the specific examples. Afterwards,  Jeanjean \cite{Jean1} gave a more general version for the the family of functionals with unconstrained problem. Very recently, Chang et al. \cite{BCJS,CJS}  extend this skill to prescribed mass problems on a Hilbert space $X$. In this paper, let $X:=H^1(\R^N)$ and the  the monotonicity trick  is given as follows:

\begin{theorem}\label{theorem2.1}(\cite{BCJS,CJS})
Let $(X,\|\cdot\|)$ be a Hilbert space and $\mathbb{I}\subset\R_+$ an interval. Consider the following family of $\mathcal{C}^2$-functional on $X$:
$$
\mathcal{I}_\mu(v)=A(v)-\mu B(v),\    \mu\in\mathbb{I}
$$
with $B(v)\geq0$ and either $A(v)\rightarrow+\infty$ or $B(v)\rightarrow+\infty$ as $\|v\|\rightarrow\infty$. Furthermore, assume that there are two points $v_1,v_2$ (independent of $\mu$) in $\mathcal{S}_a$ such that
\begin{equation*}
c_\mu=\inf_{\gamma\in{\Gamma_\mu}}\max_{t\in[0,1]}\mathcal{I}_\mu(\gamma(t))>\max\{\mathcal{I}_\mu(v_1),\mathcal{I}_\mu(v_2)\}\ \ \text{for all}\,\,\mu\in\mathbb{I}
\end{equation*}
where
$
\Gamma_\mu=\{\gamma\in \mathcal{C}([0,1],\mathcal{S}_a) : \gamma(0)=v_1, \gamma(1)=v_2\}.
$
Then for almost every $\mu\in\mathbb{I}$, there is a sequence $\{v_n\}\subset \mathcal{S}_a$ such that
\begin{itemize}
\item[(i)] $\{v_n\}$ is bounded in $X$;

\item[(ii)] $\mathcal{I}_\mu(v_n)\rightarrow c_\mu$;

\item[(iii)] $\mathcal{I}_\mu|'_{\mathcal{S}_a}(v_n)\rightarrow0$ in the dual $X^*$ of $X$.
\end{itemize}
 \end{theorem}
\section{Minimizer problem for $2<p<2+\frac{4}{N}<4+\frac{4}{N}\leq q<2\cdot2^*$}
In this section, we will show some minimization problems with combined nonlinearity and $2<p<2+\frac{4}{N}<4+\frac{4}{N}\leq q<2\cdot2^*$.
\begin{lemma}\label{lemma3.1}
Suppose that $2<p<4+\frac{4}{N}<q<2\cdot2^*$. Then
  $m(a)=-\infty$ for all $a>0$.
   \end{lemma}
\begin{proof}
Fixed $v\in \mathcal{S}_a$, since $2<p<2+\frac{4}{N}<4+\frac{4}{N}<q<2\cdot2^*$,
from \x{e1}, we get
\begin{equation*}
\aligned
\Phi(v_t)&=\frac{t^2}{2}\int_{\R^N}\frac{1+2t^Nf^2(v)}{1+2f^2(v)}|\nabla v|^2\text{d}x-\frac{t^{\frac{N(p-2)}{2}}}{p}\int_{\R^N}|f(v)|^p\text{d}x-\frac{t^{\frac{N(q-2)}{2}}}{q}\int_{\R^N}|f(v)|^q\text{d}x\\
&\rightarrow-\infty,\,\, \text{as}\,\, t\rightarrow+\infty.
\endaligned
\end{equation*}
Owing to $v\in \mathcal{S}_a$, then $v_t\in\mathcal{S}_a$ for all $t>0$. Thus it follows that $-\infty\leq m(a)\leq \Phi(v_t)$ for all $t>0$. This implies that $m(a)=-\infty$, by letting $t\rightarrow+\infty$.
\end{proof}

\begin{lemma}\label{lemma3.2}
Suppose that $2<p<4+\frac{4}{N}=q<2\cdot2^*$. Then there exists a constant $a^*_N>0$ such that
  $-\infty<m(a)$ for all $0<a<\bar{a}^*_N$ and  $\Phi$  is coercive on $\mathcal{S}_a$.
   \end{lemma}
\begin{proof}
For any $v\in H^1(\R^N)$, by Lemma \ref{lemma2.1}-(10) and (1),  one has
\begin{equation}\label{b}
\|\nabla f^2(v)\|^2_2\leq 2\|\nabla v\|^2_2\  \text{and}\  \|\nabla f(v)\|^2_2\leq \|\nabla v\|^2_2.
\end{equation}
By Lemma \ref{lemma1.2} and \x{b}, for any $v\in H^1(\R^N)$,  we get
\begin{equation}\label{b1}
\aligned
\|f(v)\|^r_r&\leq C^r_{N,r}\|\nabla f^2(v)\|^{\frac{(r-2)N}{N+2}}_2\|f(v)\|^{\frac{4N-(N-2)r}{N+2}}_2\\
&\leq C^r_{N,r} 2^{\frac{(r-2)N}{2(N+2)}}\|\nabla  v \|^{\frac{(r-2)N}{N+2}}_2\|f(v)\|^{\frac{4N-(N-2)r}{N+2}}_2, \ \ \forall\,r\in(2,2\cdot2^*).
\endaligned
\end{equation}
From Lemma \ref{lemma1.1} and \x{b}, for any $v\in H^1(\R^N)$, one has
\begin{equation}\label{bbb1}
\aligned
\|f(v)\|^s_s&\leq C^s_{N,s}\|\nabla f(v)\|^{\frac{(s-2)N}{2}}_2\|f(v)\|^{\frac{2s-N(s-2)}{2}}_2\\
&\leq C^s_{N,s} \|\nabla  v \|^{\frac{(s-2)N}{2}}_2\|f(v)\|^{\frac{2s-N(s-2)}{2}}_2, \ \ \forall\,s\in(2,2^*).
\endaligned
\end{equation}
For any $v\in \mathcal{S}_a$,  from   \x{b}-\x{bbb1}, we deduce  that
\begin{equation}\label{bb1}
\aligned
\Phi(v)&\geq\frac{1}{2}\|\nabla v\|^2_2-\frac{C^p_{N,p}}{p}\|\nabla f^2(v)\|^{\frac{(p-2)N}{N+2}}_2\|f(v)\|^{\frac{4N-(N-2)p}{N+2}}_2\\
&\quad\quad\quad\quad\quad\quad\quad-\frac{C^q_{N,q}}{q}\|\nabla f^2(v)\|^{\frac{(q-2)N}{N+2}}_2\|f(v)\|^{\frac{4N-(N-2)q}{N+2}}_2\\
&=\left(\frac{1}{2}-\frac{2C^{4+\frac{4}{N}}_{N,4+\frac{4}{N}}}{4+\frac{4}{N}}a^{\frac{2}{N}}\right)\|\nabla v\|^2_2-\frac{C^p_{N,p}2^{\frac{(p-2)N}{2(N+2)}}}{p}a^{\frac{4N-(N-2)p}{2(N+2)}}\|\nabla v\|^{\frac{(p-2)N}{N+2}}_2,
\endaligned
\end{equation}
which shows that  there is a constant
$
\bar{a}^*_N=\left(\frac{N+1}{NC^{4+\frac{4}{N}}_{N,4+\frac{4}{N}}}\right)^{\frac{N}{2}}>0
$
such that $m(a)>-\infty$ for all $0<a<\bar{a}^*_N$. This also implies that $\Phi(v)$  is coercive on $\mathcal{S}_a$, since $2<p<4+\frac{4}{N}$.
\end{proof}
Let
\begin{equation}\label{b4}
\Upsilon(v):=\frac{1}{4}\|\nabla f^2(v)\|_2^2-\frac{1}{4+\frac{4}{N}}\int_{\R^N}|f(v)|^{4+\frac{4}{N}}\text{d}x,\ \ \ \ \ \forall\,v\in \mathcal{S}_a.
\end{equation}
By \x{b4} and \x{b1}, we deduce that
\begin{equation}\label{b3}
\Upsilon(v)\geq\left[\frac{1}{4}-\frac{C^{4+\frac{4}{N}}_{N,4+\frac{4}{N}}}{4+\frac{4}{N}}a^{\frac{2}{N}}\right]\|\nabla f^2(v)\|^{2}_2
\geq0,\  \textrm{for all $0<a\leq a^*_N$.}
\end{equation}

\begin{lemma}\label{lemma3.3}
Assume that $2<p<2+\frac{4}{N}<4+\frac{4}{N}=q<2\cdot2^*$. Then
  $m(a)<0$ for all $0<a<\bar{a}^*_N$.
   \end{lemma}
\begin{proof}
For any $v\in\mathcal{S}_a$, from  \x{b4}, we get
\begin{equation}\label{b5}
\aligned
\Phi(v_t)&=\frac{t^2}{2}\int_{\R^N}\frac{1+2t^Nf^2(v)}{1+2f^2(v)}|\nabla v|^2\text{d}x-\frac{t^{\frac{N(p-2)}{2}}}{p}\|f(v)\|^p_p-\frac{t^{N+2}}{4+\frac{4}{N}}\|f(v)\|^{4+\frac{4}{N}}_{4+\frac{4}{N}}\\
&=\frac{t^2}{2}\|\nabla f(v)\|_2^2+ t^{N+2} \left(\frac{1}{4}\|\nabla f^2(v)\|_2^2-\frac{1}{4+\frac{4}{N}}\|f(v)\|^{4+\frac{4}{N}}_{4+\frac{4}{N}}\right)-\frac{t^{\frac{N(p-2)}{2}}}{p}\|f(v)\|^p_p\\
&=\frac{t^2}{2}\|\nabla f(v)\|_2^2+ t^{N+2}\Upsilon(v)-\frac{t^{\frac{N(p-2)}{2}}}{p}\|f(v)\|^p_p.
\endaligned
\end{equation}
Since $2<p<2+\frac{4}{N}$,
By \x{b3} and \x{b5}, there exists $t_0>0$ small enough such that $\Phi(v_{t_0})<0$ for all $0<a<\bar{a}^*_N$.
Hence $m(a)\leq \Phi(v_{t_0})<0$ for all $0<a<\bar{a}^*_N$.
\end{proof}

\begin{lemma}\label{lemma3.4}
Assume  that $2<p<2+\frac{4}{N}<4+\frac{4}{N}=q<2\cdot2^*$ and $0<a_1, a_2<\bar{a}^*_N$. Then

$(i)$ the following sub-additivity inequality:
\begin{equation*}
m(a_1)\leq m(a_2)+m(a_1-a_2)\ \ \text{for all}\ \ a_2\in(0,a_1)
\end{equation*}
holds and the mapping $a\mapsto m(a)$ is nonincreasing on $(0,+\infty)$;

$(ii)$  if  $m(a_2)$ or $m(a_1-a_2)$ can be attained, then
\begin{equation*}\label{b6}
m(a_1)<m(a_2)+m(a_1-a_2)\ \ \text{for all}\ \ a_2\in(0,a_1);
\end{equation*}

 $(iii)$ the mapping $a\mapsto m(a)$ is continuous  on $(0,+\infty)$.
\end{lemma}
\begin{proof}
$(i)$  Let $a>0$ be fixed. Now, we choose $\{v_n\}\subset \mathcal{S}_a$ such that $\Phi (v_n)\rightarrow m(a)<0$. For any $\alpha>1$, we set
$
\bar{v}_n(x):=v_n\left(\alpha^{-\frac{1}{N}}x\right).
$
By a simple calculation, one has
\begin{equation}\label{b7}
\aligned
&\|f(\bar{v}_n)\|^2_2=\alpha\|f(v_n)\|^2_2=\alpha a, \ \ \|f(\bar{v}_n)\|^i_i=\alpha\|f(v_n)\|^i_i,\ \ \|\nabla \bar{v}_n\|^2_2=\alpha^{1-\frac{2}{N}}\|\nabla v_n\|^2_2.
\endaligned
\end{equation}
where $i=p\ \text{or}\  q$.
From \x{b7} and $\alpha>1$, we have
\begin{equation}\label{b8}
\aligned
m(\alpha a)&\leq\Phi(\bar{v}_n)\\
&=\alpha\left[\frac{1}{2}\alpha^{-\frac{2}{N}}\|\nabla v_n\|^2_2 -\frac{1}{p}\|f(v_n)\|^p_p-\frac{1}{q}\|f(v_n)\|^q_q\right]\\
&<\alpha\Phi (v_n)=\alpha m(a)+o_n(1).
\endaligned
\end{equation}
Thus $m(\alpha a)\leq\alpha m(a)$ for any $\alpha>1$ and any $a>0$. Thus
it follows  that
\begin{equation}\label{mb8}
\aligned
m(a_1)&=\frac{a_1-a_2}{a_1} m\left(\frac{a_1}{a_1-a_2}(a_1-a_2)\right)+\frac{a_2}{a_1} m\left(\frac{a_1}{a_2}a_2\right)\\
&\leq\frac{a_1-a_2}{a_1}m\left(\frac{a_1}{a_1-a_2}(a_1-a_2)\right)+m(a_2)\\
&\leq m\left(a_1-a_2\right)+m(a_2)
\ \ \text{for all}\ \ a_2\in(0,a_1).
\endaligned
\end{equation}
This shows that  $(i)$  holds.

$(ii)$
Since $m(a_2)$ or $m(a_1-a_2)$ can be attained at $v$, we choose $v_n\equiv v$ in \x{b8}, and thus
\begin{equation}\label{b9}
\text{$m(\alpha a_2)<\alpha m(a_2)$  or $m(\alpha(a_1-a_2))<\alpha m(a_1-a_2)$ for any $\alpha>1$.}
\end{equation}
From this fact and \x{mb8} with strict inequality, we can deduce the result.

$(iii)$ Suppose that $\{a_n\}\subset(0,\bar{a}^*_N)$ satisfies $a_n\rightarrow a\in(0,\bar{a}^*_N)$ as $n\rightarrow\infty$. By the fact that $m(a)<0$ for all $0<a<\bar{a}^*_N$, if $a_n<a$ for $n$ large enough, then it follows from $(i)$ that for any $\varepsilon>0$,
 \begin{equation}\label{b10}
m(a)\leq m(a-a_n)+m(a_n)\leq m(a_n)+\varepsilon.
 \end{equation}
If $a_n\geq a$ for $n$ large enough, then we can choose a sequence $\{v_n\}\subset \mathcal{S}_{a_n}$ such that
$
 \Phi (v_n)\leq m(a_n)+\frac{1}{n}\leq\frac{1}{n}.
$
From this and \x{bb1}, we can deduce that  $\{\|\nabla v_n\|^2_2\}$ is bounded.
From \x{b1}, \x{bbb1} and Lemma \ref{lemma1.2}, it follows that
$\{\|f(v_n)\|^p_p\}$  and $\{\|f(v_n)\|^q_q\}$ are bounded.
Set
$
 \varpi_n(x):=f^{-1}\left(\sqrt{\frac{a}{a_n}}f(v_n(x))\right).
$
Then $\varpi_n\in \mathcal{S}_a$. By \x{e} with $\theta=1$, one has
\begin{equation}\label{b11}
\aligned
m(a)&\leq\Phi(\varpi_n)=\Phi(v_n)+[\Phi(\varpi_n)-\Phi(v_n)]\\
&=\Phi(v_n)+\frac{\frac{a}{a_n}-1}{2}\int_{\R^N}\frac{\left[1+2\left(\frac{a}{a_n}+1\right)f^2(v_n)\right]}{1+2f^2(v_n)}|\nabla v_n|^2\text{d}x\\
&\quad\quad\quad\quad-\frac{\left(\sqrt{\frac{a}{a_n}}\right)^{p}-1}{p}\int_{\R^N}|f(v_n)|^p\text{d}x
-\frac{\left(\sqrt{\frac{a}{a_n}}\right)^{q}-1}{q}\int_{\R^N}|f(v_n)|^q\text{d}x\\
&\leq\Phi(v_n)+o_n(1)\\
&\leq m(a_n)+o_n(1).
\endaligned
\end{equation}
From\x{b10} and \x{b11}, for any $\varepsilon>0$,  we get
\begin{equation}\label{b12}
m(a)\leq m(a_n)+o_n(1)+\varepsilon.
\end{equation}
Moreover, for the above $\varepsilon>0$, there exists $v\in \mathcal{S}_a$ such that
\begin{equation}\label{b13}
\Phi(v)<m(a)+\varepsilon.
\end{equation}
Let
$
\omega_n(x)=f^{-1}\left(\sqrt{\frac{a_n}{a}}f(v)\right).
$
 Then $\omega_n\in \mathcal{S}_{a_n}$. It follows from $\Phi(\omega_n)\rightarrow\Phi(v)$ as $n\rightarrow\infty$ and \x{b13} that
\begin{equation}\label{b14}
m(a_n)\leq\Phi(\omega_n)=\Phi(v)+[\Phi(\omega_n)-\Phi(v)]=\Phi(v)+o_n(1)<m(a)+\varepsilon+o_n(1).
\end{equation}
By \x{b12}, \x{b14} and the arbitrariness of $\varepsilon>0$, we have that  $m(a_n)\rightarrow m(a)$ as $n\rightarrow\infty$.
The proof is completed.
\end{proof}

\begin{lemma}\label{lemma3.5}
Assume  that $2<p<2+\frac{4}{N}<4+\frac{4}{N}=q<2\cdot2^*$.
Then
$m(a)<0$ is achieved for all $0<a<\bar{a}^*_N$.
\end{lemma}
\begin{proof}
Suppose that  $\{v_n\}\subset \mathcal{S}_a$ is such that $\Phi (v_n)\rightarrow m(a)$ as $n\rightarrow\infty$. It follows from the fact that $\Phi$ is coercive on $\mathcal{S}_a$ that $\{\|\nabla v_n\|^2_2\}$ is bounded.   Thus $\{v_n\}$ is bounded in $H^1(\R^N)$.

{\bf Claim: }
\begin{equation*}
\delta:=\limsup_{n\rightarrow\infty}\sup_{y\in\R^N}\int_{B_1(y)}|v_n|^2\text{d}x>0.
\end{equation*}

If $\delta=0$, by Lions' concentration compactness principle in \cite{MW},  then we know that $v_n\rightarrow 0$ in $L^r(\R^N)$ with $r\in(2,2^*)$. So $f(v_n)\rightarrow 0\  \text{in}\ \ L^{\bar{r}}(\R^N)$ for all $\bar{r}\in(2,2\cdot2^*)$ due to Lemma 2.2-$(ii)$ in \cite{CRSZ}.  Then
$
0\leq\lim\limits_{n\rightarrow\infty}\frac{1}{2}\left(\int_{\R^N}|\nabla v_n|^2\text{d}x\right)=m(a)<0,
$
which is a contradiction. Thus $\delta>0$.

Going if necessary, there is a sequence $\{y_n\}\subset\R^N$ such that
$
\int_{B_1(y_n)}|v_n|^2\text{d}x>\frac{\delta}{2}.
$
Let $\tilde{v}_n(x):=v_n(x+y_n)$. Then
$
\int_{B_1(0)}|\tilde{v}_n|^2\text{d}x>\frac{\delta}{2}.
$
Up to a subsequence, there exists $\tilde{v}_0\in H^1(\R^N)\backslash\{0\}$ such that
\begin{equation*}
\tilde{v}_n\rightharpoonup \tilde{v}_0\ \text{in}\ H^1(\R^N),\  \tilde{v}_n\rightarrow \tilde{v}_0\  \text{in}\ \ L^r_{loc}(\R^N) \  \text{for all $r\in(2,2^*)$}, \  \tilde{v}_n\rightarrow \tilde{v}_0\ \text{a.e. on $\R^N$}.
\end{equation*}
By the weaker semi-continuous of the norm, we get $\|f(\tilde{v}_0)\|^2_2\leq\liminf\limits_{n\rightarrow\infty}\|f(\tilde{v}_n)\|^2_2=a$. Next, we claim that $\|f(\tilde{v}_0)\|^2_2=a$. In fact, by a contradiction, we assume that $b:=\|f(\tilde{v}_0)\|^2_2<a$, where $b>0$.
Let $w_n:=\tilde{v}_n-\tilde{v}_0$. By a similar method as Lemma 3.2 in \cite{CRSZ}, we also obtain that
 \begin{equation}\label{section11}
\|f(\tilde{v}_n)\|^2_2-\|f(w_n)\|^2_2=\|f(\tilde{v}_0)\|^2_2+o_n(1)
 \end{equation}
and
\begin{equation}\label{section12}
\Phi (\tilde{v}_n)-\Phi (w_n) =\Phi (\tilde{v}_0)+o_n(1).
\end{equation}

From \x{section11}, we have $c_n:=\|f(w_n)\|^2_2\rightarrow a-b$ as $n\rightarrow\infty$. By \x{section12}, one has
\begin{equation*}
m(a)=\Phi (w_n)+\Phi (\tilde{v}_0)+o_n(1)\geq m(c_n)+\Phi (\tilde{v}_0)+o_n(1).
\end{equation*}
Since $\Phi(\tilde{v}_0)\geq m(b)$, suppose that $\Phi (\tilde{v}_0)>m(b)$,
by Lemma \ref{lemma3.4}, we get
\begin{equation}\label{section14}
m(a)\geq m(a-b)+\Phi (\tilde{v}_0)>m(a-b)+m(b)\geq m(a).
\end{equation}
  This is a contradiction. Thus $\Phi(\tilde{v}_0)=m(b)$ and so $m(b)$ is achieved at $\tilde{v}_0$.  Arguing as \x{section14}, we also get a contradiction and so $\|f(\tilde{v}_0)\|^2_2=a$. This shows that $f(\tilde{v}_n-\tilde{v}_0)\rightarrow0$ in $L^2(\R^N)$. By H\"{o}lder inequality, for any $\bar{r}\in(2,2\cdot2^*)$, we get that $f(\tilde{v}_n-\tilde{v}_0)\rightarrow0$ in $L^{\bar{r}}(\R^N)$ and
$
\int_{\R^N}|f(\tilde{v}_n)|^{\bar{r}}\text{d}x\rightarrow\int_{\R^N}|f(\tilde{v}_0)|^{\bar{r}}\text{d}x
$
as $n\rightarrow\infty$, where the second limit needs to use
\begin{equation*}
\|f(\tilde{v}_n)\|^{\bar{r}}_{\bar{r}}-\|f(\tilde{v}_n-\tilde{v}_0)\|^{\bar{r}}_{\bar{r}}=\|f(\tilde{v}_0)\|^{\bar{r}}_{\bar{r}}+o_n(1).
\end{equation*}
Hence $\Phi(\tilde{v}_0)\leq\liminf\limits_{n\rightarrow\infty}\Phi(\tilde{v}_n)=m(a)$, by the weak semi-continuity of the norm. Since $\tilde{v}_0\in \mathcal{S}_a$, it follows that $\Phi (\tilde{v}_0)=m(a)$.
\end{proof}
\par\noindent
{\bf Proof of Theorem \ref{theorem1.1}.}\,\, The proof is therefore complete by Lemmas \ref{lemma3.1}-\ref{lemma3.5}. $\hfill\Box$

\section{A local minimizer and a mountain pass type solution for $2<p<2+\frac{4}{N}<4+\frac{4}{N}<q\leq2^*$}
In this part, motivated by \cite{ZP}, we shall use the monotonicity trick on a constraint set developed by \cite{CJS,Jean1,MS1,MS2} to show that problem \x{p} has a local  minimizer and a mountain pass type solution for the different range on $p,q$.
\subsection{A local minimizer}
In this part, we shall show the existence of a local minimizer for $2<p<2+\frac{4}{N}<4+\frac{4}{N}<q\leq2^*$. To this aim, inspired by \cite{Chen3,JSLa}, let the function $\mathcal{H}_a: (0,+\infty)\rightarrow\R$ be
\begin{equation}\label{e3}
\aligned
\mathcal{H}_a(t)&=\frac{1}{2}-\frac{C^p_{N,p}}{p}2^{\frac{N(p-2)}{2(N+2)}}a^{\frac{4N-p(N-2)}{2(N+2)}}t^{\frac{Np-\left(4N+4\right)}{2(N+2)}}-\frac{C^q_{N,q}}{q}2^{\frac{N(q-2)}{2(N+2)}}a^{\frac{4N-q(N-2)}{2(N+2)}}t^{\frac{Nq-\left(4N+4\right)}{2(N+2)}}\\
&=:\frac{1}{2}-A_1a^{\frac{4N-p(N-2)}{2(N+2)}}t^{\frac{Np-\left(4N+4\right)}{2(N+2)}}-A_2a^{\frac{4N-q(N-2)}{2(N+2)}}t^{\frac{Nq-\left(4N+4\right)}{2(N+2)}},
\endaligned
\end{equation}
where
$A_1=\frac{C^p_{N,p}}{p}2^{\frac{N(p-2)}{2(N+2)}}$ and $A_2=\frac{C^q_{N,q}}{q}2^{\frac{N(q-2)}{2(N+2)}}$.

\begin{lemma}\label{lemma4.1}
Assume that  $2<p<2+\frac{4}{N}<4+\frac{4}{N}<q\leq2^*$ holds. For each  $a>0$, the function $\mathcal{H}_a(t)$ has a unique global maximum at
\begin{equation*}
\bar{t}_a:=\left[\frac{A_1\left(4N+4-Np\right)}{A_2\left(Nq-4N-4\right)}\right]^{\frac{2(N+2)}{N(q-p)}}a^{\frac{N-2}{N}}>0
\end{equation*}
 and there exists $a^*_N>0$ such that  the maximum satisfies
\begin{displaymath}
\max_{t\in(0,+\infty)}\mathcal{H}_a(t)=\mathcal{H}_a(\bar{t}_a)\left\{
\begin{array}{ll}
>0, & \textrm{if $a<a^*_N$},\\
=0,& \textrm{if $a=a^*_N$},\\
<0,& \textrm{if $a>a^*_N$}.
\end{array} \right.\\
\end{displaymath}
\end{lemma}
\begin{proof}
By \x{e3}, we deduce that
\begin{equation*}
\aligned
\frac{\text{d}\mathcal{H}_a(t)}{\text{d}t}
&=A_1{\frac{\left[\left(4N+4\right)-Np\right]}{2(N+2)}}a^{\frac{4N-p(N-2)}{2(N+2)}}t^{\frac{Np-\left(4N+4\right)}{2(N+2)}-1}\\
&\quad\quad\quad\quad\quad\quad\quad\quad-A_2{\frac{\left[Nq-\left(4N+4\right)\right]}{2(N+2)}}a^{\frac{4N-q(N-2)}{2(N+2)}}t^{\frac{Nq-\left(4N+4\right)}{2(N+2)}-1}.
\endaligned
\end{equation*}
It follows from $\frac{\text{d}\mathcal{H}_a(t)}{\text{d}t}=0$ that there exists
$
\bar{t}_a:=\left[\frac{A_1\left(4N+4-Np\right)}{A_2\left(Nq-4N-4\right)}\right]^{\frac{2(N+2)}{N(q-p)}}a^{\frac{N-2}{N}}>0
$
such that $\bar{t}_a$ is a unique global maximum of $\mathcal{H}_a(t)$ on $(0,+\infty)$, since $\mathcal{H}_a(t)\rightarrow-\infty$ as $t\rightarrow0^+$ and $\mathcal{H}_a(t)\rightarrow-\infty$ as $t\rightarrow+\infty$. Furthermore, its global maximum value is given by
 \begin{equation*}
\aligned
&\max_{t\in(0,+\infty)}\mathcal{H}_a(t)\\
&=\mathcal{H}_a(\bar{t}_a)\\
&=\frac{1}{2}-A_1a^{\frac{4N-p(N-2)}{2(N+2)}}\bar{t}_a^{\frac{Np-\left(4N+4\right)}{2(N+2)}}-A_2a^{\frac{4N-q(N-2)}{2(N+2)}}\bar{t}_a^{\frac{Nq-\left(4N+4\right)}{2(N+2)}}\\
&=\frac{1}{2}-\left\{A_1 \left[\frac{A_1\left(4N+4-Np\right)}{A_2\left(Nq-4N-4\right)}\right]^{\frac{Np-4N-4}{2(q-p)}}+A_2 \left[\frac{A_1\left(4N+4-Np\right)}{A_2\left(Nq-4N-4\right)}\right]^{\frac{Nq-4N-4}{2(q-p)}}\right\}a^{\frac{2}{N}}.
 \endaligned
 \end{equation*}
 Thus there exists $a^*_N>0$ such that  the conclusion holds.
\end{proof}

\begin{lemma}\label{lemma4.2}
Assume that  $2<p<2+\frac{4}{N}<4+\frac{4}{N}<q\leq2^*$ holds. Then for any $a>0$ and  $v\in\mathcal{S}_a$,
\begin{equation*}
\Phi_\theta(v)\geq\|\nabla v\|^2_2\mathcal{H}_a\left(\|\nabla v\|^2_2\right),\ \ \ \ \ \forall\,\theta\in[1/2,1].
\end{equation*}
\end{lemma}
\begin{proof}
By Lemma \ref{lemma1.1}, \x{b} and \x{b1},  for each $v\in \mathcal{S}_a$, we have
 \begin{equation*}
\aligned
\Phi_\theta(v)&\geq\frac{1}{2}\|\nabla v\|^2_2-\frac{1}{p}\|f(v)\|^p_p-\frac{1}{q}\|f(v)\|^{q}_{q}\\
&\geq\frac{1}{2}\|\nabla v\|^2_2-\frac{C^p_{N,p}}{p}2^{\frac{N(p-2)}{2(N+2)}}a^{\frac{4N-p(N-2)}{2(N+2)}}\|\nabla  v \|^{\frac{N(p-2)}{N+2}}_2-\frac{C^q_{N,q}}{q}2^{\frac{N(q-2)}{2(N+2)}}a^{\frac{4N-q(N-2)}{2(N+2)}}\|\nabla  v \|^{\frac{N(q-2)}{N+2}}_2\\
&\geq\|\nabla v\|^2_2\left[\frac{1}{2}-A_1a^{\frac{4N-p(N-2)}{2(N+2)}}\|\nabla  v \|^{\frac{Np-4N-4}{N+2}}_2-A_2a^{\frac{4N-q(N-2)}{2(N+2)}}\|\nabla  v \|^{\frac{Nq-4N-4}{N+2}}_2\right]\\
&=\|\nabla v\|^2_2\mathcal{H}_a\left(\|\nabla v\|^2_2\right),\ \ \ \ \forall\,\theta\in[1/2,1].
 \endaligned
 \end{equation*}
 The proof is completed.
\end{proof}

Let $t_0:=\bar{t}_{a^*_N}>0$. Define
$$
\Lambda_{k}:=\left\{ v\in H^1(\R^N): \|\nabla v\|^2_2<k\right\}
$$
and
$$
\sigma_\theta(a):=\inf\limits_{v\in \mathcal{S}_a\cap\Lambda_{t_0}}\Phi_\theta(v).
$$

\begin{lemma}\label{lemma4.3}
Assume that  $2<p<2+\frac{4}{N}<4+\frac{4}{N}<q\leq2^*$ holds. Then, for any $a\in (0,a^*_N)$, the following conclusion holds
\begin{equation*}
\sigma_\theta(a)=\inf_{v\in \mathcal{S}_a\cap\Lambda_{t_0}}\Phi_\theta(v)<0<\inf_{v\in \partial(\mathcal{S}_a\cap\Lambda_{t_0})}\Phi_\theta(v),\ \ \,\forall\,\theta\in[1/2,1].
\end{equation*}
\end{lemma}
\begin{proof}
Let $v\in \mathcal{S}_a$ be fixed. Then
$
v_s(x):=f^{-1}(s^{N/2}f(v(sx)))\in  \mathcal{S}_a,\,\forall\,s>0,
$
and
\begin{equation*}
\Phi_\theta(v_s)=\frac{s^2}{2}\int_{\R^N}\frac{1+2s^Nf^2(v)}{1+2f^2(v)}|\nabla v|^2\text{d}x-\frac{\theta s^{\frac{N(p-2)}{2}}}{p}\int_{\R^N}|f(v)|^p\text{d}x-\frac{\theta s^{\frac{N(q-2)}{2}}}{q}\int_{\R^N}|f(v)|^q\text{d}x.
\end{equation*}
Thus there exists $0<s_0<1$ small enough such that
\begin{equation*}
\|\nabla v_{s_0}\|^2_2=s^2_0 \int_{\R^N}\frac{1+2s^N_0f^2(v)}{1+2f^2(v)}|\nabla v|^2\text{d}x\leq s^2_0 \|\nabla v\|^2_2<t_0.
\end{equation*}
and $\Phi_\theta(v_{s_0})<0$.
Hence it follows that  $v_{s_0}\in \mathcal{S}_a\cap\Lambda_{t_0}$ and $\Phi_\theta(v_{s_0})<0$ for $s_0>0$ small enough and  any $\theta\in[1/2,1]$. So $\sigma_\theta(a)\leq\Phi_\theta(v_{s_0})<0,\,\forall\,\theta\in[1/2,1]$. By this fact and Lemmas \ref{lemma4.1}-\ref{lemma4.2}, we can get the conclusion.
\end{proof}

\begin{remark}\label{remark4.1}
By Lemma \ref{lemma4.3}, we deduce that $\sigma_1(a)\leq\sigma_\theta(a)\leq\sigma_{1/2}(a)$ for all  $\theta\in[1/2,1]$ and $a\in(0,a^*_N)$.
\end{remark}

\begin{lemma}\label{lemma4.4}
Assume that  $2<p<2+\frac{4}{N}<4+\frac{4}{N}<q\leq2^*$  and $a\in (0,a^*_N)$. Then, for all $b\in(0,a)$ and $\theta\in[1/2,1]$, $\sigma_\theta(a)\leq \sigma_\theta(b)+\sigma_\theta(a-b)$; and if $\sigma_\theta(b)$ or $\sigma_\theta(a-b)$ can be achieved, then the inequality is strict.
\end{lemma}
\begin{proof}
Let $b\in(0,a)$ be fixed. For any $\iota\in \left[1,\frac{a}{b}\right]$, it follows from  \x{e3} and Lemma \ref{lemma4.1} that
\begin{equation}\label{e6}
\aligned
\mathcal{H}_b\left(\frac{b\iota}{a}t_0\right)&=\frac{1}{2}-A_1b^{\frac{4N-p(N-2)}{2(N+2)}}\left(\frac{b\iota}{a}t_0\right)^{\frac{Np-\left(4N+4\right)}{2(N+2)}}
-A_2b^{\frac{4N-q(N-2)}{2(N+2)}}\left(\frac{b\iota}{a}t_0\right)^{\frac{Nq-\left(4N+4\right)}{2(N+2)}}
\\
&=\frac{1}{2}-A_1a^\frac{4N-p(N-2)}{2(N+2)}\left(t_0\right)^{\frac{Np-\left(4N+4\right)}{2(N+2)}}\left(\frac{b\iota}{a}\right)^{\frac{Np-\left(4N+4\right)}{2(N+2)}+\frac{4N-p(N-2)}{2(N+2)}}\iota^{-\frac{4N-p(N-2)}{2(N+2)}}\\
&\quad\quad-A_2a^\frac{4N-p(N-2)}{2(N+2)}\left(t_0\right)^{\frac{Nq-\left(4N+4\right)}{2(N+2)}}\left(\frac{b\iota}{a}\right)^{\frac{Nq-\left(4N+4\right)}{2(N+2)}+\frac{4N-q(N-2)}{2(N+2)}}\iota^{-\frac{4N-q(N-2)}{2(N+2)}}\\
&=\frac{1}{2}-A_1a^\frac{4N-p(N-2)}{2(N+2)}\left(t_0\right)^{\frac{Np-\left(4N+4\right)}{2(N+2)}}\left(\frac{b\iota}{a}\right)^{\frac{p-2}{N+2}}\iota^{-\frac{4N-p(N-2)}{2(N+2)}}\\
&\quad\quad-A_2a^\frac{4N-p(N-2)}{2(N+2)}\left(t_0\right)^{\frac{Nq-\left(4N+4\right)}{2(N+2)}}\left(\frac{b\iota}{a}\right)^{\frac{q-2}{N+2}}\iota^{-\frac{4N-q(N-2)}{2(N+2)}}\\
&\geq\mathcal{H}_a(t_0)=\mathcal{H}_a(\bar{t}_{a^*_0})>\mathcal{H}_{a^*_0}(\bar{t}_{a^*_0})=0.
\endaligned
\end{equation}
This shows that $\mathcal{H}_b(t)>0$ for any $t\in\left[\frac{b}{a}t_0,t_0\right]$.
Since $\sigma_\theta(b)<0$, it follows that there exists a sequence $\{v_n\}\subset\mathcal{S}_b\cap\Lambda_{t_0}$  such that
$
\|\nabla v_n\|^2_2\mathcal{H}_b\left(\|\nabla v_n\|^2_2\right)\leq\Phi_\theta(v_n)<0 \  \text{for $n$ large enough},
$
which, together with \x{e6}, implies that for $n$ large enough,
$$
\|\nabla v_n\|^2_2<\frac{b}{a}t_0.
$$
For $\tau\in \left(1,\frac{a}{b}\right]$, let
$
\ddot{v}_n(x):=v_n\left(\tau^{-\frac{1}{N}}x\right).
$
By a simple calculation, we get that
\begin{equation*}
\|\nabla\ddot{v}_n\|^2_2=\tau^{\frac{N-2}{N}}\|\nabla v_n\|^2_2<\tau^{\frac{N-2}{N}}\frac{b}{a}t_0<t_0.
\end{equation*}
Similar to \x{b8}, we can infer that
$
\sigma_\theta(\theta b)\leq\tau \sigma_\theta(b),\,\forall\, \tau\in \left(1,\frac{a}{b}\right].
$
Thus
\begin{equation*}
\aligned
\sigma_\theta(a)&\leq \frac{a-b}{a}\sigma_\theta\left(\frac{a}{a-b}(a-b)\right)+\frac{b}{a}\sigma_\theta\left(\frac{a}{b}b\right)\\
&\leq \sigma_\theta\left(a-b\right)+\sigma_\theta\left(b\right)
\endaligned
\end{equation*}
and the inequality is strict if $\sigma_\theta\left(a-b\right)$ or $\sigma_\theta\left(b\right)$ can be achieved.
\end{proof}

\begin{lemma}\label{lemma4.5}
Assume that  $2<p<2+\frac{4}{N}<4+\frac{4}{N}<q\leq2^*$ holds. Then the function $a\mapsto \sigma_\theta(a)$ is continuous on $(0,a^*_N)$, for any $\theta\in[1/2,1]$.
\end{lemma}
\begin{proof}
The proof is similar to the proof of Lemma \ref{lemma3.4}-$(iii)$. So we omit the proof.
\end{proof}

\begin{lemma}\label{lemma4.6}
Assume that  $2<p<2+\frac{4}{N}<4+\frac{4}{N}<q\leq2^*$ holds. Then for any $\theta\in[1/2,1]$, the function $\sigma_\theta(a)$ is achieved on $(0,a^*_N)$ and $\Phi_\theta|'_{\mathcal{S}_a}=0$.
\end{lemma}
\begin{proof}
 Let $\{v_n\}\subset\mathcal{S}_a\cap \Lambda_{t_0} $ be a minimizing sequence of $\sigma_\theta(a)$.  By Lemma \ref{lemma4.3}, we have
\begin{equation}\label{e-6}
\|\nabla v_n\|^2_2<t_0,\ \ \ \|f(v_n)\|^2_2=a,\ \ \Phi_\theta(v_n)=\sigma_\theta(a)+o_n(1)<0\ \ \text{for large $n$}.
\end{equation}
By Lemma \ref{lemma2.1} and \x{b1}, one has
\begin{equation*}
\aligned
\int_{\R^N}|v_n|^2\text{d}x&\leq C\int_{|v_n|\leq1}|f(v_n)|^2\text{d}x+\int_{|v_n|\geq1}|v_n|^4\text{d}x\\
&\leq C\|f(v_n)\|^2_2+C\|\nabla v_n\|_2^{\frac{2N}{N+2}}\|f(v_n)\|^{\frac{8}{N+2}}_2\leq C.
\endaligned
\end{equation*}

Let
$$
\bar{\delta}:=\limsup\limits_{n\rightarrow\infty}\sup\limits_{y\in\R^N}\int_{B_1(y)}|v_n|^2\text{d}x.
$$
If $\bar{\delta}=0$, then by Lions' concentration compactness principle, we have that $v_n\rightarrow0$ in $L^r(\R^N)$ for all $r\in(2,6)$. Thus for any $2<p<2+\frac{4}{N}<4+\frac{4}{N}<q\leq2^*$, it follows from Lemma \ref{lemma2.3} that
\begin{equation}\label{e-7}
\int_{\R^N}|f(v_n)|^p\text{d}x\rightarrow0\ \ \text{and}\ \ \int_{\R^N}|f(v_n)|^q\text{d}x\rightarrow0\ \ \text{as}\ \ n\rightarrow\infty.
\end{equation}
From \x{e-6}, \x{e-7} and Sobolev inequality, we get
 \begin{equation*}
 \aligned
\sigma_\theta(a)+o_n(1)&=\frac{1}{2}\|\nabla v_n\|^2_2-\frac{1}{p}\|f(v_n)\|^p_p-\frac{1}{q}\|f(v_n)\|^{q}_{q}\\
&=\frac{1}{2}\|\nabla v_n\|^2_2+o_n(1)\\
&\geq o_n(1),
\endaligned
 \end{equation*}
which is a contradiction, since $\sigma_\theta(a)<0$. Thus $\bar{\delta}>0$.

Up to a subsequence, there exists a sequence $\{y_n\}\subset\R^N$ such that
$$
\int_{B_1(y_n)}|v_n|^2\text{d}x>\frac{\delta}{2}.
$$
Let $\hat{\upsilon}_n(x):=v_n(x+y_n)$. Then there exists $\hat{\upsilon}_\theta\in H^1(\R^N)\setminus\{0\}$ such that, passing to subsequence,
\begin{equation*}
\hat{\upsilon}_n\rightharpoonup\hat{\upsilon}_\theta\ \text{in $H^1(\R^N)$,}\ \ \hat{\upsilon}_n\rightarrow\hat{\upsilon}_\theta\ \text{in $L^r_{loc}(\R^N)$ for all $r\in (2,2^*)$},
\ \hat{\upsilon}_n\rightarrow\hat{\upsilon}_\theta\ \text{a.e. on $\R^N$.}
\end{equation*}
From \x{e-6}, we deduce that
\begin{equation}\label{9-10}
\|\nabla \hat{\upsilon}_n\|^2_2<t_0,\ \ \ 0<\|f(\hat{\upsilon}_\theta)\|^2_2\leq\liminf_{n\rightarrow\infty}\|f(\hat{\upsilon}_n)\|^2_2=a,\ \ \Phi_\theta(\hat{\upsilon}_n)=\sigma_\theta(a)+o_n(1).
\end{equation}

Set $\tilde{w}_n:=\hat{\upsilon}_n-\hat{\upsilon}_\theta$. By \x{9-10} and Lemma 5.7 in \cite{CRSZ}, we know that
 \begin{equation}\label{9-11}
 \Phi_\theta(\hat{\upsilon}_n)= \Phi_\theta(\hat{\upsilon}_\theta)+ \Phi_\theta(\tilde{w}_n)+o_n(1),
 \end{equation}
 \begin{equation}\label{9-12}
\|\nabla\hat{\upsilon}_n\|^2_2=\|\nabla\hat{\upsilon}_\theta\|^2_2+\|\nabla \tilde{w}_n \|^2_2+o_n(1),
 \end{equation}
 and
\begin{equation}\label{9-13}
\|f(\tilde{w}_n)\|^2_2=\|f(\hat{\upsilon}_n)\|^2_2-\|f(\hat{\upsilon}_\theta)\|^2_2+o_n(1)=a-\|f(\hat{\upsilon}_\theta)\|^2_2+o_n(1).
 \end{equation}
Next, we claim that $\|f(\tilde{w}_n)\|^2_2\rightarrow0$ as $n\rightarrow\infty$. In fact, let $\bar{c}:=\|f(\hat{\upsilon}_\theta)\|^2_2\leq a$. If $\bar{c}=a$, then the claim holds. Suppose that $\bar{c}<a$. In view of \x{9-12}, \x{9-13} and \x{9-10}, for $n$ large enough, one has
\begin{equation}\label{9-14}
\beta_n:=\|f(\tilde{w}_n)\|^2_2\leq a,\ \  \ \|\nabla \tilde{w}_n \|^2_2\leq\|\nabla\hat{\upsilon}_n\|^2_2<t_0.
\end{equation}
By \x{9-14}, we have
\begin{equation}\label{9-15}
\tilde{w}_n\in \mathcal{S}_{\beta_n}\cap\Lambda_{t_0}, \ \ \  \Phi_\theta(\tilde{w}_n)\geq \sigma_\theta(\beta_n):=\inf_{v\in\mathcal{S}_{\beta_n}\cap\Lambda_{t_0}} \Phi_\theta(v).
\end{equation}
From \x{9-10}, \x{9-11} and \x{9-15}, we have
\begin{equation}\label{9-16}
\aligned
\sigma_\theta(a)+o_n(1)&= \Phi_\theta(\hat{\upsilon}_n)=\ \Phi_\theta(\hat{\upsilon}_\theta)+ \Phi_\theta(\tilde{w}_n)+o_n(1)\\
&\geq \Phi_\theta(\hat{\upsilon}_\theta)+\sigma_\theta(\beta_n)+o_n(1).
\endaligned
\end{equation}
By virtue of \x{9-16}, Lemma \ref{lemma4.3} and \x{9-13}, we have
\begin{equation}\label{9-17}
\sigma_\theta(a)\geq \Phi_\theta(\hat{\upsilon}_\theta)+\sigma_\theta(a-\bar{c}).
\end{equation}
Moreover, by \x{9-14} and the weak lower semi-continuity of norm, we conclude that $\hat{\upsilon}_\theta\in \mathcal{S}_{\bar{c}}\cap \overline{\Lambda_{t_0}}$. Thus $ \Phi_\theta(\hat{\upsilon}_\theta)\geq \sigma_\theta(\bar{c})$. If $ \Phi_\theta(\hat{\upsilon}_\theta)> \sigma_\theta(\bar{c})$, then it follows from \x{9-17} and Lemma \ref{lemma4.4} that
$
\sigma_\theta(a)>\sigma_\theta(\bar{c})+\sigma_\theta(a-\bar{c})\geq \sigma_\theta(a),
$
which is a contradiction. Hence $\sigma_\theta(\bar{c})$ is achieved at $\hat{\upsilon}_\theta$. By the  strict inequality in Lemma \ref{lemma4.4}, it follows from \x{9-17} that
$
\sigma_\theta(a)\geq \sigma_\theta(\bar{c})+\sigma_\theta(a-\bar{c})>\sigma_\theta(a).
$
This is a contraction. So $\|f(\hat{\upsilon}_\theta)\|^2_2=a$ and $\hat{\upsilon}_\theta\in \mathcal{S}_{a}\cap \overline{\Lambda_{t_0}}$. Thus the claim holds. From the claim, we infer that $\|f(\tilde{w}_n)\|^p_p\rightarrow0$ as $n\rightarrow\infty$, for all $2<p<2+\frac{4}{N}$. Since $\Phi_\theta(\hat{\upsilon}_\theta)\geq \sigma_\theta(a)$, it follows from \x{9-10} and \x{9-11} that
\begin{equation*}
\aligned
\sigma_\theta(a)+o_n(1)&=\Phi_\theta(\hat{\upsilon}_n)=\Phi_\theta(\hat{\upsilon}_\theta)+\Phi_\theta(\tilde{w}_n)+o_n(1)\\
&\geq  \sigma_\theta(a)+\frac{1}{2}\|\nabla \tilde{w}_n\|^2_2-\frac{1}{p}\|f(\tilde{w}_n)\|^p_p-\frac{1}{q}\|f(\tilde{w}_n)\|^{q}_{q}\\
&\geq \sigma_\theta(a)+\frac{1}{2}\|\nabla \tilde{w}_n\|^2_2+o_n(1),
\endaligned
\end{equation*}
which implies that $\|\nabla \tilde{w}_n\|^2_2\rightarrow0$ as $n\rightarrow\infty$, which, together with $\|f(\tilde{w}_n)\|^2_2\rightarrow0$ as $n\rightarrow\infty$, implies that
$
\int_{\R^N}(|\nabla \tilde{w}_n|^2+f^{2}(\tilde{w}_n))\text{d}x\rightarrow0\  \text{as $n\rightarrow\infty$}.
$
By Lemma \ref{lemmas2.4}, we have that $\tilde{w}_n\rightarrow0$ in $H^1(\R^N)$, that is, $\hat{\upsilon}_n\rightarrow\hat{\upsilon}_\theta$ in $H^1(\R^N)$. Hence
$
\|f(\hat{\upsilon}_\theta)\|^2_2=a,\  \|\nabla\hat{\upsilon}_\theta\|^2_2\leq t_0,\ \sigma_\theta(\bar{c})=\sigma_\theta(a)=\Phi_\theta(\hat{\upsilon}_\theta),
$
and  by the standard Schwartz rearrangement, it follows that $\hat{\upsilon}_\theta$ is radial. 
From Lemma \ref{lemma4.3}, we deduce that $\|\nabla\hat{\upsilon}_\theta\|^2_2<t_0$. By Corollary 2.4 in \cite{Chen3}, we infer that $\Phi_\theta|'_{\mathcal{S}_a}(\hat{\upsilon}_\theta)=0$.
\end{proof}

\par\noindent
{\bf Proof of Theorem \ref{theorem1.2}-$(i)$.}\,  By Lemmas \ref{lemma4.1}-\ref{lemma4.6}, we only need to choose $\theta=1$ and so \x{p} has a local radial minimizer $\hat{\upsilon}_1\in \mathcal{S}_a$ such that $\Phi(\hat{\upsilon}_1)=\sigma_1(a)$ and $\Phi|'_{\mathcal{S}_a}(\hat{\upsilon}_1)=0$.  Thus there exists $\lambda_*\in\R$ such that $\Phi'(\hat{\upsilon}_1)-\lambda_*  f(\hat{\upsilon}_1)f'(\hat{\upsilon}_1)=0$ in $(H^1(\R^N))^*$ and similar to the proof of Lemma \ref{lemma4.8}, one has $\lambda_* <0$.  $\hfill\Box$

\subsection{Mountain pass type solution}
In this subsection,  we study $\Phi(v)$ on radial space $H^1_r(\R^N)$.  From this and Palais' symmetric principle in \cite{MW}, the critical point in $H^1_r(\R^N)$ is also the critical point in $H^1(\R^N)$. To this end, let
$
\mathcal{S}^r_a=\mathcal{S}_a\cap H^1_r(\R^N).
$
Now,
we shall used Jeanjean's trick skills to show the existence of mountain pass type normalized solution for $2<p<2+\frac{4}{N}<4+\frac{4}{N}<q\leq2^*$, which is different from the local minimizer.
\begin{lemma}\label{lemma4.7}
Assume that  $2<p<2+\frac{4}{N}<4+\frac{4}{N}<q\leq2^*$ holds and $a\in(0,a^*_N)$. Then for any  $\theta\in[1/2,1]$, there exist  $v_1,v_2\in\mathcal{S}^r_a$   independent of $\theta$ such that $\|\nabla v_1\|^2_2<t_0$, $\|\nabla v_2\|^2_2>t_0$ and
\begin{equation*}
c_\theta(a):=\inf_{\gamma\in\Gamma}\max_{t\in[0,1]}\Phi_\theta(\gamma(t))>\max\left\{\Phi_\theta(v_1),\Phi_\theta(v_2)\right\},
\end{equation*}
where
\begin{equation*}
\Gamma=\left\{\gamma\in\mathcal{C}([0,1], \mathcal{S}^r_a)\big| \gamma(0)=v_1, \gamma(1)=v_2\right\}.
\end{equation*}
\end{lemma}
\begin{proof}
Let $v_1:=\hat{\upsilon}_\theta$, where $\hat{\upsilon}_\theta$ was obtained in Lemma \ref{lemma4.6}. By Lemma \ref{lemma4.6}, we deduce that $\|\nabla v_1\|^2_2<t_0$ and $\sigma_\theta(a)=\Phi_\theta(v_1)<0$.

Fixed $\omega\in \mathcal{S}^r_a$, defined by
$
\omega_t=:f^{-1}\left(t^{N/2}f(\omega(tx))\right),
$
then $\omega_t\in\mathcal{S}^r_a$. It follows that
\begin{equation*}
\aligned
\Phi_\theta(\omega_t)&=\frac{t^2}{2}\int_{\R^N}\frac{1+2t^Nf^2(\omega)}{1+2f^2(\omega)}|\nabla \omega|^2\text{d}x- \theta t^{-N}\int_{\R^N}\left[\frac{1}{p}\left|t^{N/2}f(\omega)\right|^{p}+\frac{1}{q}\left|t^{N/2}f(\omega)\right|^{q}\right]\text{d}x\\
&\leq\frac{t^2}{2}\int_{\R^N}\frac{1+2t^Nf^2(\omega)}{1+2f^2(\omega)}|\nabla \omega|^2\text{d}x-\frac{t^{\frac{N(p-2)}{2}}}{2p}\int_{\R^N} \left|f(\omega)\right|^{p}\text{d}x-\frac{t^{\frac{N(q-2)}{2}}}{2q}\int_{\R^N} \left|f(\omega)\right|^{q}\text{d}x\\
&\rightarrow-\infty, \ \ \text{as}\ \ t\rightarrow+\infty
\endaligned
\end{equation*}
and
\begin{equation*}
\|\nabla \omega_t\|^2_2=t^2\int_{\R^N}\frac{1+2t^N f^2(\omega)}{1+2f^2(\omega)}|\nabla \omega|^2\text{d}x\geq t^{2} \|\nabla \omega\|^2_2\rightarrow+\infty,\ \ \text{as}\ \ t\rightarrow+\infty.
\end{equation*}
Thus there exists $t^*>1$ sufficiently large such that $\Phi_\theta(\omega_{t^*})<\Phi_\theta(v_1)$ and $\|\nabla \omega_{t^*}\|^2_2>t_0$. Let $v_2=\omega_{t^*}$. Then $\|\nabla v_2\|^2_2>t_0$ and $\Phi_\theta(v_2)<\Phi_\theta(v_1)$ for any $\theta\in[1/2,1]$.
Let 
\begin{equation*}
\gamma_*(t) =f^{-1}\left(\frac{\sqrt{a}}{\|f\left((1-t)v_1 + tv_2\right)\|_2}[f\left((1-t)v_1 + tv_2\right)]\right), \quad t \in [0, 1].
\end{equation*}
Then $\gamma_*\in\Gamma$ and thus $\Gamma\neq\emptyset$.
Moreover, for any $\gamma\in \Gamma$, it follows from the definition of $\Gamma$ that $\|\nabla\gamma(0)\|^2_2=\|\nabla v_1\|^2_2<t_0$ and $\|\nabla \gamma(1)\|^2_2=\|\nabla v_2\|^2_2>t_0$. By the continuity of $\gamma$, we can get that there exists $t_*\in(0,1)$ such that
$\|\nabla \gamma(t_*)\|^2_2=t_0$. Thus for any $\gamma\in\Gamma$, one has
$$
\max\limits_{t\in[0,1]}\Phi_\theta(\gamma(t))\geq \Phi_\theta(\gamma(t_*))\geq \inf\limits_{v\in \partial(\mathcal{S}^r_a\cap\Lambda_{t_0})}\Phi_\theta(v),
$$
which, together with Lemma \ref{lemma4.3}, implies that
\begin{equation*}
c_\theta(a):=\inf_{\gamma\in\Gamma}\max_{t\in[0,1]}\Phi_\theta(\gamma(t))\geq \inf_{v\in \partial(\mathcal{S}^r_a\cap\Lambda_{t_0})}\Phi_\theta(v)>0>\sigma_\theta(a)=\inf_{v\in \mathcal{S}^r_a\cap\Lambda_{t_0}}\Phi_\theta(v),
\end{equation*}
for any $\theta\in[1/2,1]$. This completed the proof.
\end{proof}
\begin{remark}\label{remark4.2}
By Lemma \ref{lemma4.7}, we deduce that $\sigma_\theta(a)<0<c_\theta(a)\leq c_{1/2}(a)$ for any fixed $\theta\in[1/2,1]$ and $a\in(0,a^*_N)$.
\end{remark}

\begin{lemma}\label{lemma4.8}
Assume that  $2<p<2+\frac{4}{N}<4+\frac{4}{N}<q\leq2^*$ holds and  $a\in(0,a^*_N)$.  For almost every $\theta\in[1/2,1]$, there exists $v_\theta\in \mathcal{S}^r_a$, $\lambda_\theta<0$ such that $\Phi_\theta(v_\theta)=c_\theta(a)$ and $\Phi'_\theta(u_\theta)-\lambda_\theta  f(v_\theta)f'(v_\theta)=0$ in $(H^1(\R^N))^*$.
\end{lemma}
\begin{proof}
By Lemma \ref{lemma4.7} and Theorem \ref{theorem2.1}, for almost every $\theta\in[1/2,1]$, we can infer that there exists a bounded Palais-Smale sequence $\{v_n\}\subset \mathcal{S}^r_a$ such that
\begin{equation}\label{eee14}
\Phi_\theta(v_n)\rightarrow c_\theta(a)\ \ \text{and}\ \ \Phi_\theta|'_{\mathcal{S}^r_a}(v_n)\rightarrow0.
\end{equation}
By the boundedness of $\{v_n\}$,   there exists $v_\theta\in H^1_r(\R^N)$ such that
\begin{equation}\label{e10}
v_n\rightharpoonup v_\theta \ \ \text{in $H^1_r(\R^N)$}, \ \ v_n\rightarrow  v_\theta\ \text{in $L^s(\R^N)$}\ \ \text{for all $s\in(2,2^*)$},\ \ v_n\rightarrow v_\theta\ \ \text{a.e. on $\R^N$.}
\end{equation}
Moreover following Lemma 3.2 in \cite{ZLW} or Breestycki-Lions (see Lemma 3 in \cite{BL}), we also know that there exists $\lambda_n\in\R$ such that
\begin{equation}\label{eee15}
\Phi'_\theta(v_n)-\lambda_n f(v_n)f'(v_n)\rightarrow0\ \text{ in} \ (H^1(\R^N))^*.
\end{equation}
It follows from \x{eee14} that
 \begin{equation}\label{e12}
c_\theta(a)+o_n(1)=\frac{1}{2}\|\nabla v_n\|^2_2-\frac{\theta}{p}\int_{\R^N}|f(v_n)|^p\text{d}x-\frac{\theta}{q}\int_{\R^N}|f(v_n)|^q\text{d}x.
 \end{equation}
Since
\begin{equation*}
\aligned
\langle\Phi'_\theta(v_n),f'(v_n)/f(v_n)\rangle&=\int_{\R^N}|\nabla v_n|^2\text{d}x+\int_{\R^N}\frac{2f^2(v_n)}{1+2f^2(v_n)}|\nabla v_n|^2\text{d}x\\
&\quad\quad\quad\quad\quad\quad-\theta\int_{\R^N}|f(v_n)|^p\text{d}x-\theta\int_{\R^N}|f(v_n)|^q\text{d}x,
\endaligned
\end{equation*}
it follows from \x{e12} that $\{\langle\Phi'_\theta(v_n),f'(v_n)/f(v_n)\rangle\}$ is bounded, which together with \x{eee15}, imply that $\{\lambda_n\}$ is bounded in $\R$. Up to a subsequence, there exists $\lambda_\theta\in\R$ such that $\lambda_n\rightarrow \lambda_\theta$ as $n\rightarrow\infty$. By Lemma \ref{lemma2.3} and \x{e10}, we have that
$f(v_n)\rightarrow  f(v_\theta)$ in $L^s(\R^N)$ for all $s\in(2,2^*)$.

Next, we claim the following conclusions hold.

(i) $\Phi'_\theta(v_n)-\lambda_\theta  f(v_n)f'(v_n)\rightarrow0$ in $(H^1(\R^N))^*$;

(ii) $\Phi'_\theta(u_\theta)-\lambda_\theta  f(v_\theta)f'(v_\theta)=0$ in $(H^1(\R^N))^*$;

(iii)  $\lambda_\theta<0$;

(iv) $v_n\rightarrow v_\theta$ in $H^1(\R^N)$.

Now, we prove the claim. It is easy to prove that (i) and (ii) hold.
Next, we prove that $\lambda_\theta<0$. In fact, in (ii) testing with $\frac{f(v_\theta)}{f'(v_\theta)}$, we  get
\begin{equation}\label{e14}
\aligned
0=\langle\Phi'_\theta(v_\theta)-\lambda_\theta f(v_\theta)f'(v_\theta), f(v_\theta)/f'(v_\theta) \rangle&=\int_{\R^N}|\nabla v_\theta|^2\text{d}x+\int_{\R^N}\frac{2f^2(v_\theta)}{1+2f^2(v_\theta)}|\nabla v_\theta|^2\text{d}x\\
&\quad\quad-\lambda_\theta \int_{\R^N}|f(v_\theta)|^2\text{d}x-\theta \int_{\R^N}|f(v_\theta)|^p\text{d}x\\
&\quad\quad-\theta \int_{\R^N}|f(v_\theta)|^q\text{d}x=0.
\endaligned
\end{equation}
Moreover, since $\mathcal{P}_{\theta}(v_\theta)=0$, by \x{e14} and \x{e2}, we have that
\begin{equation*}
\aligned
0&=\langle\Phi'_\theta(v_\theta)-\lambda_\theta f(v_\theta)f'(v_\theta), f(v_\theta)/f'(v_\theta) \rangle-\mathcal{P}_{\theta}(v_\theta)\\
&=-\frac{N-2}{2}\int_{\R^N}\frac{2f^2(v_\theta)}{1+2f^2(v_\theta)}|\nabla v_\theta|^2\text{d}x-\lambda_\theta\int_{\R^N}|f(v_\theta)|^2\text{d}x\\
&\quad\quad+\theta\left[\frac{N(p-2)}{2p}-1\right]\int_{\R^N}|f(v_\theta)|^p\text{d}x+\theta\left[\frac{N(q-2)}{2q}-1\right]\int_{\R^N}|f(v_\theta)|^q\text{d}x,
\endaligned
\end{equation*}
which shows that
\begin{equation*}
\aligned
\lambda_\theta\int_{\R^N}|f(v_\theta)|^2\text{d}x
&=-\frac{N-2}{2}\int_{\R^N}\frac{2f^2(v_\theta)}{1+2f^2(v_\theta)}|\nabla v_\theta|^2\text{d}x\\
&\quad+\theta\left[\frac{N(p-2)}{2p}-1\right]\int_{\R^N}|f(v_\theta)|^p\text{d}x+\theta\left[\frac{N(q-2)}{2q}-1\right]\int_{\R^N}|f(v_\theta)|^q\text{d}x\\
&<0.
\endaligned
\end{equation*}
Hence $\lambda_\theta<0$ and so  (iii)  holds.

By (i), (ii) and the boundedness of $\{v_n\}$, we deduce that
\begin{equation}\label{e15}
\langle\Phi'_\theta(v_n)-\lambda_\theta  f(v_n)f'(v_n),v_n-v_\theta\rangle=o_n(1)
\end{equation}
 and
\begin{equation}\label{e16}
\langle\Phi'_\theta(v_\theta)-\lambda_\theta  f(v_\theta)f'(v_\theta),v_n-v_\theta\rangle=o_n(1).
\end{equation}
 From \x{e15} and \x{e16}, we have
\begin{equation}\label{e17}
\aligned
0&=\langle\Phi'_\theta(v_n)-\lambda_\theta  f(v_n)f'(v_n),v_n-v_\theta\rangle-\langle\Phi'_\theta(v_\theta)-\lambda_\theta  f(v_\theta)f'(v_\theta),v_n-v_\theta\rangle\\
&=\int_{\R^N}|\nabla v_n-\nabla v_\theta|^2\text{d}x-\lambda_\theta\int_{\R^N}\left( f(v_n)f'(v_n)- f(v_\theta)f'(v_\theta)\right)(v_n-v_\theta)\text{d}x\\
&\quad\quad-\int_{\R^N}\left[|f(v_n)|^{p-2}f(v_n)f'(v_n)- |f(v_\theta)|^{p-2}f(v_\theta)f'(v_\theta)\right](v_n-v_\theta)\text{d}x\\
&\quad\quad-\int_{\R^N}\left[|f(v_n)|^{q-2}f(v_n)f'(v_n)- |f(v_\theta)|^{q-2}f(v_\theta)f'(v_\theta)\right](v_n-v_\theta)\text{d}x.
\endaligned
\end{equation}
By a similar argument as Lemma 3.11 in \cite{FS}, we may prove that there is a constant
$C>0$ such that
 \begin{equation}\label{e18}
 \aligned
\int_{\R^N}|\nabla v_n-\nabla v_\theta|^2\text{d}x&-\lambda_\theta\int_{\R^N}\left(f(v_n)f'(v_n)- f(v_\theta)f'(v_\theta)\right)(v_n-v_\theta)\text{d}x\\
&\geq C \int_{\R^N}|\nabla v_n-\nabla v_\theta|^2\text{d}x-\lambda_\theta\int_{\R^N}|v_n-v_\theta|^2\text{d}x.
\endaligned
 \end{equation}
 By \x{e10} and H\"{o}lder inequality, one has
 \begin{equation}\label{e19}
\int_{\R^N}\left[|f(v_n)|^{p-2}f(v_n)f'(v_n)- |f(v_\theta)|^{p-2}f(v_\theta)f'(v_\theta)\right](v_n-v_\theta)\text{d}x=o_n(1)
 \end{equation}
 and
  \begin{equation}\label{e20}
\int_{\R^N}\left[|f(v_n)|^{q-2}f(v_n)f'(v_n)- |f(v_\theta)|^{q-2}f(v_\theta)f'(v_\theta)\right](v_n-v_\theta)\text{d}x=o_n(1).
 \end{equation}
 From \x{e17}-\x{e20}, we deduce that $v_n\rightarrow v_\theta$ in $H^1(\R^N)$, which shows that (iv) holds. Thus $\Phi_\theta(v_\theta)=c_\theta(a)$ and $\Phi'_\theta(u_\theta)-\lambda_\theta  f(v_\theta)f'(v_\theta)=0$ in $(H^1(\R^N))^*$.
 The proof is completed.
\end{proof}
\par\noindent
{\bf Proof of Theorem \ref{theorem1.2}-$(ii)$.}\,
By Lemma \ref{lemma4.8}, $v_\theta$ is solution of \x{subsection1} with $\lambda_\theta<0$ for almost every $\theta\in[1/2,1]$, namely, $\Phi_\theta(v_\theta)=c_\theta(a)$ and $\Phi'_\theta(u_\theta)-\lambda_\theta  f(v_\theta)f'(v_\theta)=0$ in $(H^1(\R^N))^*$. Since $\Phi(v)\leq\Phi_\theta(v)\leq\Phi_{1/2}(v)$ for all $v\in H^1_r(\R^N)$, it follows that
\begin{equation}\label{e21}
c_1(a)\leq \Phi_\theta(v_\theta)=c_\theta(a)\leq c_{1/2}(a).
\end{equation}

Now, we choose a  sequence $\theta_n\rightarrow 1^-$ such that $v_{\theta_n}\in \mathcal{S}^r_a$, $\lambda_{\theta_n}<0$,
\begin{equation}\label{e22}
\Phi_{\theta_n}(v_{\theta_n})=c_{\theta_n}(a)
\end{equation}
 and
\begin{equation}\label{e23}
\text{$\Phi'_{\theta_n}(u_{\theta_n})-\lambda_{\theta_n}  f(v_{\theta_n})f'(v_{\theta_n})=0$ in $(H^1(\R^N))^*$.}
\end{equation}
By Lemma \ref{lemmas2.1} and \x{e23}, we know that $\mathcal{P}_{\theta_n}(v_{\theta_n})=0$. By this, \x{e21}, \x{e22}  and Lemma \ref{lemma1.1}, we conclude that
\begin{equation*}
\aligned
c_{1/2}(a)&\geq c_{\theta_n}=\Phi_{\theta_n}(v_{\theta_n})-\frac{2}{N(q-2)}\mathcal{P}_{\theta_n}(v_{\theta_n)}\\
&=\frac{q-\left(4+\frac{4}{N}\right)}{2(q-2)}\int_{\R^N}|\nabla v_{\theta_n}|^2\text{d}x+\frac{1}{q-2}\int_{\R^N}|\nabla f(v_{\theta_n})|^2\text{d}x-\frac{(q-p)\theta_n}{p(q-2)} \|f(v_{\theta_n})\|^p_p\\
&\geq\frac{q-\left(4+\frac{4}{N}\right)}{2(q-2)}\int_{\R^N}|\nabla v_{\theta_n}|^2\text{d}x-\frac{q-p}{p(q-2)}C^p_{N,p}a^{\frac{2p-N(p-2)}{4}}\left(\int_{\R^N}|\nabla v_{\theta_n}|^2\text{d}x\right)^{\frac{N(p-2)}{4}},
\endaligned
\end{equation*}
which, together with  $2<p<2+\frac{4}{N}<4+\frac{4}{N}<q\leq2^*$, imply that $\{\|\nabla v_{\theta_n}\|^2_2\}$ is bounded. Combining this with $\|f(v_{\theta_n})\|^2_2=a$, it follows from Lemma \ref{lemmas2.4} that
 $\{v_{\theta_n}\}$ is bounded in $H^1_r(\R^N)$. Moreover, there exists $\lambda^*\leq0$ such that $\lambda_n\rightarrow \lambda^*$ as $n\rightarrow\infty$.
Furthermore, one has
\begin{equation}\label{e24}
\aligned
\lim_{n\rightarrow\infty}\Phi(v_{\theta_n})&=\lim_{n\rightarrow\infty}\left[\Phi_{\theta_n}(v_{\theta_n})+\frac{\theta_n-1}{p}\int_{\R^N}|f(v_{\theta_n})|^p\text{d}x+\frac{\theta_n-1}{q}\int_{\R^N}|f(v_{\theta_n})|^q\text{d}x\right]\\
&=\lim_{n\rightarrow\infty} \Phi_{\theta_n}(v_{\theta_n})=\lim_{n\rightarrow\infty}c_{\theta_n}\leq c_{1/2}(a)
\endaligned
\end{equation}
and
\begin{equation*}
\aligned
\lim_{n\rightarrow\infty}\|\Phi|'_{\mathcal{S}^r_a}(v_{\theta_n})\|_*
&=\lim_{n\rightarrow\infty}\sup_{\int_{\R^N}f(v_{\theta_n})f'(v_{\theta_n})\psi\text{d}x,\ \|v_{\theta_n}\|=1}\bigg[\langle\Phi'(v_{\theta_n}),\psi\rangle\\
&\quad\quad\quad+\frac{\theta_n-1}{p}\int_{\R^N}|f(v_{\theta_n})|^{p-2}f(v_{\theta_n})f'(v_{\theta_n})\psi\text{d}x \\
&\quad\quad\quad+\frac{\theta_n-1}{q}\int_{\R^N}|f(v_{\theta_n})|^{q-2}f(v_{\theta_n})f'(v_{\theta_n})\psi\text{d}x\bigg]=0.
\endaligned
\end{equation*}
Hence $\{v_{\theta_n}\}$ is a bounded Palasi-Smale sequence for $\Phi$ on $\mathcal{S}^r_a$. Next, similar to the argument as Lemma \ref{lemma4.8}, there exists $v_0\in H^1_r(\R^N)$ and $\lambda^*<0$ such that $v_{\theta_n}\rightarrow v_0$ in $H^1_r(\R^N)$ and $\Phi'(v_0)-\lambda^*f(v_0)f'(v_0)=0$.
Thus $v_0$ is a solution of \x{p} for all $a\in(0,a^*_N)$ and satisfies
\begin{equation*}
c_{1/2}(a)\geq\Phi(v_0)\geq c_1(a)>0>\sigma_1(a)=\inf_{v\in \mathcal{S}^r_a\cap\Lambda_{t_0}}\Phi(v)=\Phi(\hat{\upsilon}_1).
\end{equation*}
This shows that $v_0\neq \hat{\upsilon}_1$ and $\Phi(v_0)=\lim\limits_{n\rightarrow\infty}c_{\theta_n}$. Finally, by a similar argument as Lemma 4.4  in \cite{SZ}, we also conclude that  $\lim\limits_{n\rightarrow\infty}c_{\theta_n}=c_1(a)$.
 This completes the proof.
$\hfill\Box$

\vskip4mm
\par\noindent{\large\bf Acknowledgements}\\
The research of J. Chen  and J. Sun was supported by National Natural Science Foundation of China (Grant Nos. 12361024),
and partly by supported by Jiangxi Provincial Natural Science Foundation (Nos. 20252BAC250127, 20242BAB23002, 20242BAB25001  and 20232ACB211004).
The research of J. Zhang was supported by the National Natural Science Foundation of
China (12271152), the Key project of Scientific Research Project of Department of
Education of Hunan Province (24A0430).

\end{document}